\documentclass[11pt,reqno]{amsart}
\usepackage{amsmath,amsthm,amssymb,mathrsfs,stmaryrd,color, amsfonts}
\usepackage[all, cmtip]{xy}
\usepackage{url}
\numberwithin{equation}{section}
\usepackage[utf8]{inputenc}
\usepackage[T1]{fontenc}

\usepackage{extarrows}
\usepackage{mathtools}
\usepackage{enumerate}
\usepackage[colorlinks=true,hyperindex, linkcolor=blue, pagebackref=false, citecolor=blue,pdfpagelabels]{hyperref}
\usepackage[capitalize]{cleveref}

\newtheorem{theorem}{Theorem}[section]
\newtheorem*{theorem*}{Theorem}
\newtheorem*{definition*}{Definition}

\newtheorem{theo}[theorem]{Theorem}
\newtheorem*{theo*}{Theorem}
\newtheorem{lem}[theorem]{Lemma}
\newtheorem{prop}[theorem]{Proposition}
\newtheorem{cor}[theorem]{Corollary}

\theoremstyle{definition}

\newtheorem{example}[theorem]{Example}

\newtheorem{rem}[theorem]{Remark}

\DeclareMathOperator{\Amp}{Amp} 
\DeclareMathOperator{\Aut}{Aut}
\DeclareMathOperator{\Bl}{Bl}

\DeclareMathOperator{\Bir}{Bir}
\DeclareMathOperator{\BN}{BN}
\DeclareMathOperator{\codim}{codim}

\DeclareMathOperator{\coker}{coker}

\DeclareMathOperator{\ch}{ch}
\DeclareMathOperator{\Coh}{Coh}

\DeclareMathOperator{\Ext}{Ext}
\DeclareMathOperator{\ext}{ext}

\DeclareMathOperator{\Hom}{Hom}
\DeclareMathOperator{\Hilb}{Hilb}
\DeclareMathOperator{\id}{id}
\DeclareMathOperator{\im}{im}

\DeclareMathOperator{\Mov}{Mov}
\DeclareMathOperator{\Nef}{Nef}

\DeclareMathOperator{\NS}{NS}
\DeclareMathOperator{\Pic}{Pic}

\DeclareMathOperator{\Pos}{Pos}

\DeclareMathOperator{\rk}{rk}

\DeclareMathOperator{\Spec}{Spec}
\DeclareMathOperator{\Supp}{Supp}

\DeclareMathOperator{\Sym}{Sym}

\DeclareMathOperator{\sheafHom}{\mathcal{H}\it{om}}

\DeclareMathOperator{\sheafExt}{\mathcal{E}\it{xt}}
\DeclareMathOperator{\sheafTor}{\mathcal{T}\it{or}}

\renewcommand{\epsilon}{\varepsilon}

\newcommand{\Acal}{\mathcal{A}}
\newcommand{\Bcal}{\mathcal{B}}
\newcommand{\Ccal}{\mathcal{C}}

\newcommand{\Ecal}{\mathcal{E}}
\newcommand{\Fcal}{\mathcal{F}}

\newcommand{\Hcal}{\mathcal{H}}
\newcommand{\Ical}{\mathcal{I}}
\newcommand{\Lcal}{\mathcal{L}}

\newcommand{\Ocal}{\mathcal{O}}

\newcommand{\Qcal}{\mathcal{Q}}
\newcommand{\Scal}{\mathcal{S}}
\newcommand{\Tcal}{\mathcal{T}}

\newcommand{\Xcal}{\mathcal{X}}

\newcommand{\Wcal}{\mathcal{W}}

\newcommand{\Q}{\mathbb{Q}}
\newcommand{\R}{\mathbb{R}}
\newcommand{\Z}{\mathbb{Z}}
\newcommand{\C}{\mathbb{C}}

\newcommand{\Pbb}{\mathbb{P}}

\newcommand{\mfrak}{\mathfrak{m}}

\usepackage{color}
\newcounter{commentcounter}

\def\?{\ {\bf\color{red}???}\ 
\immediate\write16{}
\immediate\write16{Warning: There was still a question mark . . . }
\immediate\write16{}}

\begin{document}

\title[Birational geometry of the Mukai system of rank two and genus two]{Birational geometry of the Mukai system of rank two and genus two}
	\author[I.~Hellmann]{Isabell Hellmann}
	\thanks{The author is supported by the SFB/TR 45 `Periods, Moduli Spaces and Arithmetic of Algebraic varieties' of the DFG (German Research Foundation) and the Bonn International Graduate School.}
	\address{Mathematisches Institut, Universit\"at Bonn, Endenicher Allee 60, 53115 Bonn, Germany}
	\email{igb@math.uni-bonn.de}
	
	\begin{abstract}
	Using the techniques of Bayer--Macr\`i, we determine the walls in the movable cone of the Mukai system of rank two for a general K3 surface $S$ of genus two. We study the (essentially unique) birational map to $S^{[5]}$ and decompose it into a sequence of flops. We give an interpretation of the exceptional loci in terms of Brill--Noether loci.
	\end{abstract}
	
	\maketitle
	
	\setcounter{tocdepth}{2}
	
\section{Introduction}
Let $(S,H)$ be a polarized K3 surface of genus two, that is a double covering $\pi \colon S \rightarrow \Pbb^2$ ramified over a smooth sextic curve and $H =\Ocal_S(1) = \pi^*\Ocal_{\Pbb^2}(1)$. Moreover, assume that $\Pic(S)= \Z \cdot H$. We consider the moduli space $M=M_H(0,n,-1)$ of $H$-Gieseker stable coherent sheaves on $S$ with Mukai vector $v=(0,n,-1)$. This is an irreducible holomorphic symplectic variety of dimension $2(n^2+1)$. A point in $M_H(0,n,-1)$ corresponds to a stable sheaf $\Ecal$ on $S$ such that $\Ecal$ is pure of dimension one with support in the linear system $|nH|$ and $\chi(\Ecal)=-1$. Taking the (Fitting) support defines a Lagrangian fibration
\[ f \colon M_H(0,n,-1) \longrightarrow |nH| \cong \Pbb^{n^2+1},\ \Ecal \mapsto \Supp(\Ecal) \]
known as the \emph{Mukai system of genus two} \cite{Beau}, \cite{Mu}. Over a point of $|nH|$ which corresponds to a smooth curve $D \subset S$, the fibers of $f$ are abelian varieties isomorphic to $\Pic^{n^2-1}(D)$. So, $M_H(0,n,-1)$ can be viewed as a compactified relative Jacobian of the universal curve $\Ccal \rightarrow |nH|$.\\

Due to its rich and beautiful geometry the Mukai system serves as an interesting example in Hyperk\"ahler geometry. An instance of this, shall be demonstrated in this article, where we determine all birational models in the case $n=2$.\\

It is easy to see that $M\coloneqq M_H(0,n,-1)$ is birational to $S^{[n^2+1]}$. Namely, let $\xi \in S^{[n^2+1]}$ such that $\Supp(\xi)$ consists of $n^2+1$ points in general position. Then there is a unique smooth curve $D \in |nH|$ such that $\xi \subset D$ and this allows to define a rational map
\[ T \colon S^{[n^2+1]} \dashrightarrow M_H(0,n,-1),\ \xi \mapsto \Ocal_D(-\xi) \otimes \Ocal_S(n)|_D.\]
Conversely, a general point in $M$ is given by $\Lcal \in \Pic^{n^2-1}(D)$, for a smooth curve $D \in |nH|$ and thus generically $\dim H^0(S,\Lcal^\vee \otimes \Ocal_S(n)|_D)= 1$. Hence, $T$ is birational.
The morphism $T$ can be defined more conceptually via the spherical twist $T_{\Ocal_S(-n)} \colon D^b(S) \rightarrow D^b(S)$ \cite[\S 8.1]{FM}.
Let $\Ical_{\xi} \in M_H(1,0,-n^2)$ be the ideal sheaf of a point $\xi \in S^{[n^2+1]}$, which is contained in the open subset, where $h^0(\Ical_\xi(n))=1$. By definition, $T_{\Ocal_S(-n)}(\Ical_{\xi})$ fits into a short exact sequence
\[ 0 \rightarrow \Ocal_S(-n) \rightarrow \Ical_{\xi} \rightarrow T_{\Ocal_S(-n)}(\Ical_{\xi}) \rightarrow 0. \]
We conclude $T(\Ical_\xi) = T_{\Ocal_S(-n)}(\Ical_{\xi})(n)$. In other words, $T$ is the composition
\[ S^{[n^2+1]}= M_H(1,0,-n^2) \xrightarrow{T_{\Ocal_S(-n)}} M_\sigma(0,n,-2n^2-1) \xrightarrow{-\otimes \Ocal_S(n)} M_\sigma(0,n,-1) \dashrightarrow M_H(0,n,-1),\]
where the first two arrows are isomorphisms and $\sigma$ is a suitable stability condition. The last arrow is the birational transformation coming from wall-crossing along a path from $\sigma$ into the Gieseker chamber.\\
If $n=1$, then all curves in $|H|$ are irreducible and therefore $T_{\Ocal_S(-1)}(\Ical_\xi)$ is a stable sheaf provided that $h^0(\Ical_\xi(1))=1$. The indeterminancy of $T$ is exactly the closed subset of $\xi \in S^{[2]}$ through which passes a pencil of curves in $|H|$, which is identified with
\[ \Pbb^2 \subset S^{[2]},\ x \mapsto \pi^{-1}(x).\]
A resolution of $T$ is the original example of a Mukai flop \cite{Mu}:
\[ \xymatrix@R-1pc@C-2pc{ & \Bl_{\Pbb^2}S^{[2]} \cong \Hilb^2(\Ccal/|H|) \ar[dl]\ar[dr]&&& (\xi \subset C)\ar@{|->}[dl]\ar@{|->}[dr]\\
S^{[2]} \ar@{-->}[rr]&& M_H(0,1,-1),& \xi && \Ocal_C(-\xi)\otimes \Ocal_S(1)|_C.
}\]
For $n=2$, we show
\begin{prop}[Prop \ref{birational map M S5}]\label{indeterminancy}
The spherical twist at $\Ocal_S(-2)$ induces an isomorphism
\[ T \colon S^{[5]} \setminus \{ \xi \in S^{[5]} \mid h^0(\Ical_{\xi}(2)) \geq 2\} \xrightarrow \sim M\setminus ( \{ \Ecal \in M \mid h^0(\Ecal) \geq 1\} \cup M^0_\Delta).\]
Here, $\Delta \subset |2H|$ is the locus of non-reduced curves and $M^0_\Delta \subset M$ is the irreducible component of $f^{-1}(\Delta)$ that consists of vector bundles of rank two and degree one on the underlying reduced curve.
\end{prop}

From the birational geometer's point of view, $M_H(0,n,-1)$ and $S^{[n^2+1]}$ are extremal in the following sense. If one considers the decomposition of the movable cone
\[ \Mov(S^{[n^2+1]}) \subset \NS(S^{[n^2+1]})_\R \cong \R^2\]
into chambers corresponding to birational models, $S^{[n^2+1]}$ is at the one end, for it admits a divisorial contraction given by the Hilbert--Chow morphism and $M_H(0,n,-1)$ with the Lagrangian fibration is at the other end. For $n=1$, the Mukai system $M_H(0,1,-1)$ is the only other smooth birational model of $S^{[2]}$. If $n>1$, the presence of reducible and non-reduced curves in the linear system $|nH|$ make the situation more complicated. The paper in hand deals with the case $n=2$. We prove the following result.

\begin{theo}[Thm \ref{birational models}]\label{main theo}
Let $(S,H)$ be a polarized K3 surface with $\Pic(S)=\Z \cdot H$ and $H^2=2$. There are five (smooth) birational models of $S^{[5]}$ or $M\coloneqq M_H(0,2,-1)$, respectively. They are connected by a chain of flopping contractions
\[ \xymatrix@R-1pc@C-2pc{
& \Bl_{W_2}S^{[5]} \ar[dr]\ar[dl] && \Bl_{\tilde{W}_3}X_1\ar[dr]\ar[dl] && \Bl_{\tilde{Z}_3}X_3\ar[dr]\ar[dl] && \Bl_{Z_1}M\ar[dr]\ar[dl] \\
S^{[5]} \ar@{-->}[rr]^{g_1} && X_1 \ar@{-->}[rr]^{g_2} && X_2 && X_3 \ar@{-->}[ll]_{g_3} &&  M \ar@{-->}[ll]_{g_4}} \]
for some subvarieties $W_2 \subset W_3 \subset S^{[5]}$ such that
\begin{itemize}
\item $W_2$ is a $\Pbb^3$-bundle over $M_H(0,1,-6)$,
\item $W_3 \setminus W_2 $ is a $\Pbb^2$-bundle over an open subset of $M_H(0,1,-5)\times S$.
\end{itemize}
and subvarieties $Z_1 \subset Z_3 \subset M$ such that
\begin{itemize}
\item $Z_1$ is a $\Pbb^4$-bundle over $S$,
\item $Z_3\setminus Z_1 $ is a $\Pbb^2$-bundle over an open subset of $S^{[3]}$.
\end{itemize}
Here, $\tilde{W}_3$ (resp.\ $\tilde{Z}_3$) is the strict transform of $W_3$ (resp.\ $Z_3$) under $g_1$ (resp.\ $g_4$).
\end{theo}

We prove Theorem \ref{main theo} using the methods of Bayer--Macr\`i \cite{MMP}. Their techniques give a procedure to compute the walls in the movable cone and to identify the curves, which are contracted at every step. The exceptional loci are components of the Brill--Noether loci
\[ \BN^i(S^{[5]}) \coloneqq \{ \xi \in S^{[5]} \mid h^0(\Ical_{\xi}(2)) \geq i+1\} \subset S^{[5]},\ i=1,2 \]
and
\[ \BN^i(M) \coloneqq \{ \Ecal \in M \mid h^0(\Ecal) \geq i+1 \} \subset M,\ i=0,1.\]
More precisely, we have the following description.
\begin{prop}\label{main prop}
\begin{enumerate}[\rm (i)]
\item We have \[
W_2 = \{ \xi \in S^{[5]} \mid \text{there is}\ C \in |H|\ \text{such that}\ \xi \subset C\}\subset \BN^2(S^{[5]}),\ \text{and}\]
\[ W_3 = \{ \xi \in S^{[5]} \mid \text{there is}\ x \in \xi\ \text{and}\ C \in |H|\ \text{such that}\ \xi\setminus\{x\} \subset C\} \subset \BN^1(S^{[5]}). \]
\item $Z_1$ (resp.\ $Z_3$) is the component of $\BN^1(M)$ (resp.\ $\BN^0(M)$) that dominates the locus of smooth curves in $|2H|$.
\end{enumerate}
\end{prop}

\subsection*{Outline} The core of this paper is Section \ref{section models}, which is entirely an application of the results of \cite{MMP} to the Mukai system of rank two and genus two. In particular, we compute the walls in $\Mov(S^{[5]})$ and at each wall, we get a numerical characterization of the projective bundles that get contracted. The preceeding sections can be seen as the foundation for the geometrical interpretation of these computations. Precisely, in Section \ref{section mukai}, we collect the necessary information on the Mukai system. In Section \ref{section T}, we prove Proposition \ref{indeterminancy} by explicit considerations and likewise explicitly we study components of the appearing Brill--Noether loci leading to Proposition \ref{main prop} in Section \ref{section BN}. These components will later be identified with the exceptional loci of the transformations in Theorem \ref{main theo}.

\subsection*{Acknowledgements} This paper is part of my PhD thesis at the university of Bonn. I wish to heartily thank my advisor Daniel Huybrechts for his constant guidance and support. Further thanks go to Thorsten Beckmann and Georg Oberdieck for explaining me a lot of the techniques involved in this article. It's a pleasure to acknowledge helpful discussions with Alberto Cattaneo and Mirko Mauri. I thank them, too.

\section{The Mukai system of rank two and genus two}\label{section mukai}
We collect some results on the Mukai system of rank two and genus two. Let $(S,H)$ be a polarized K3 surface of genus $2$ such that $\Pic(S) = \Z \cdot H$. We assume that the linear system $|H|$ contains a smooth, irreducible curve, so $S$ is a double covering $\pi \colon S \rightarrow \Pbb^2$ ramified over a smooth sextic curve. We consider the moduli space $M=M_H(0,2,k)$ of $H$-Gieseker stable coherent sheaves on $S$ with Mukai vector $v=(0,2,k)$, where $k \equiv 1$ mod $2$. This is an irreducible holomorphic symplectic variety of dimension $10$. A point in $M_H(0,2,k)$ corresponds to a stable sheaf $\Ecal$ on $S$ such that $\Ecal$ is pure of dimension one with support in the linear system $|2H|$ and $\chi(\Ecal)=k$. Taking the (Fitting) support defines a Lagrangian fibration
\[ f \colon M_H(0,2,k) \longrightarrow B  \coloneqq |2H| \cong \Pbb^{5} \]
known as the \emph{Mukai system of rank two and genus two} \cite{Beau}, \cite{Mu}.\\

Since tensoring with $\Ocal_S(1)$ induces an isomorphism
\[ \tau_H \colon M_H(0,2,k) \xrightarrow\sim M_H(0,2,k+4),\]
it is immediate that the isomorphism class of $M_H(0,2,k)$ depends only on $k$ modulo 4. The following Lemma shows that actually the isomorphism class is the same for all odd $k$. Alternatively, $M_H(0,2,k)$ for odd $k$ can be characterized as the unique birational model of $S^{[5]}$ admitting a Lagrangian fibration (cf.\ Section \ref{section num walls}).

\begin{lem}\label{iso mod 2}
There is an isomorphism
\[ M_H(0,2,1) \longrightarrow M_H(0,2,-1),\ \Ecal \mapsto \Ecal^\vee \coloneqq \sheafExt^1_{\Ocal_S}(\Ecal,\Ocal_S). \]
In particular, all the moduli spaces $M_H(0,2,k)$ for odd $k$ are isomorphic.
\end{lem}

\begin{proof}
Every $\Ecal \in M_H(0,2,1)$ is pure of dimension one. Therefore, $\sheafExt^i_{\Ocal_S}(\Ecal,\Ocal_S)=0$ for $i \neq 1$ and the natural map
\[ \Ecal \xlongrightarrow \sim \Ecal^{\vee\vee} = \sheafExt^1_{\Ocal_S}(\sheafExt^1_{\Ocal_S}(\Ecal,\Ocal_S),\Ocal_S)\]
is an isomorphism, \cite[Prop 1.1.10]{HL}. One easily sees that $\Ecal^\vee$ is again $H$-Gieseker stable. 
\end{proof}

In the following, we usually choose $k= -1$ and set
\[ M \coloneqq M_H(0,2,-1).\]
With this choice of $k$, a stable vector bundle of rank two and degree one on a smooth curve $C \in |H|$ defines a point in $M$.

\subsection{The linear systems \texorpdfstring{$|H|$}{} and \texorpdfstring{$|2H|$}{}}
The geometry of the Mukai system is closely related to the structure of the curves in the linear systems $|H|$ and $|2H|$. We have
\[H^0(S,\Ocal_S(k)) \cong H^0(\Pbb^2,\Ocal_{\Pbb^2}(k))\oplus H^0(\Pbb^2,\Ocal_{\Pbb^2}(k-3)),\]
and so in particular $H^0(S,\Ocal_S(k)) \cong H^0(\Pbb^2,\Ocal_{\Pbb^2}(k))$ if $k =1,2$. We conclude that all curves in $|H|$ and $|2H|$ are double covers of lines and conics in $\Pbb^2$, respectively. A curve in $|H|$ (resp.\ $|2H|$) has geometric genus $2$ (resp.\ $5$).\\

We use the Segre map
$m \colon |H| \times |H| \rightarrow |2H|$ to define the subloci
\begin{equation*}
\Delta \coloneqq m(\Delta_{|H|}) \subset \Sigma \coloneqq \im(m) \subset |2H|.
\end{equation*}
Then $\Sigma \cong \Sym^2|H|$ is the locus of non-integral curves, it is four-dimensional and its generic member is reduced and has two smooth irreducible components in the linear system $|H|$ meeting transversally in two points (we use $\rho(S) =1$). The subset $\Delta \cong |H| \cong \Pbb^2$ is the locus of non-reduced curves.

\subsection{Fibers of the Mukai morphism and structure of \texorpdfstring{$M$}{}}
The moduli space $M$ contains a dense open subset consisting of sheaves that are line bundles on their support. The restriction of the Mukai morphism to this locus is smooth \cite[Prop 2.8]{LP} and the image of the restricted morphism is $B \setminus \Delta$ \cite[Lem 3.5.3]{CRS}. In particular, $M_\Sigma \coloneqq f^{-1}(\Sigma)$ contains a dense open subset that parametrizes the pushforwards of line bundles, but  $M_\Delta \coloneqq f^{-1}(\Delta)$ does not.\\

Following \cite[Proposition 3.7.1]{CRS}, the fibers of the Mukai morphism $f \colon M \rightarrow B$ show the following characteristics:
\begin{equation}\label{fibertype}
f^{-1}(x)=\
\begin{cases}
\text{is reduced and irreducible} & \text{if}\ x\in B\setminus\Sigma\\
\text{is reduced and has two irreducible components} & \text{if}\ x\in \Sigma\setminus\Delta \\
\text{has two irreducible components with multiplicities} & \text{if}\ x\in \Delta.
\end{cases}
\end{equation}
Let us make this more precise for generic points:
\begin{itemize}
\item In the first case, let $x \in B \setminus \Sigma$ correspond to a smooth curve $D$, then $f^{-1}(x) \cong \Pic^3(D)$.
\item In the second case, let $x \in \Sigma \setminus \Delta$ correspond to the union $D = D_1 \cup D_2$ of two smooth curves meeting transversally in two points. Then $f^{-1}(x)$ contains a dense open subset parameterizing line bundles on $D$. The two irreducible components of $f^{-1}(x)$ correspond to line bundles of partial degree $(2,1)$ and $(1,2)$. 
\item In the third case, let $x \in \Delta$ correspond to a non-reduced curve with smooth underlying curve $C \in |H|$. Then $f^{-1}(x)$ has two non-reduced irreducible components, which we denote as follows
\begin{equation}\label{fiber decomp over delta}
M_{2C} \coloneqq f^{-1}(x)_{\rm red} = M_{2C}^0 \cup M_{2C}^1.
\end{equation}
The first component  $M_{2C}^0$ consists of those sheaves, that are pushed forward from the reduced curve $C$. With its reduced structure it is isomorphic to the moduli space of stable vector bundles of rank two and degree one on $C$. The other component $M_{2C}^1$ is the closure of those sheaves that can not be endowed with an $\Ocal_C$-module structure. All these sheaves fit into a short exact sequence
\begin{equation}\label{ext}
0 \rightarrow i_*(\Lcal(x)\otimes \omega_C^{-1}) \rightarrow \Ecal \rightarrow i_*\Lcal \rightarrow 0,
\end{equation}
where $i \colon C \hookrightarrow S$ is the inclusion, $\Lcal \in \Pic^1(C)$ is the torsionfree part of $\Ecal|_C$ and $x \in C$ is the support of the torsion part of $\Ecal|_C$. This extension is intrinsically associated to $\Ecal$, for details see \cite{He}.
\end{itemize}

The decomposition \eqref{fiber decomp over delta} also exists globally
\begin{equation}\label{comp of M_Delta} M_{\Delta} = f^{-1}(\Delta)_{\rm red} = M^0_{\Delta} \cup M^{1}_{\Delta}, \end{equation}
with $M_{\Delta}^0$ being a relative moduli space of stable vector bundles and $M_{\Delta}^1$ the closure of its complement \cite[Prop 3.7.23]{CRS}.

\section{The birational map \texorpdfstring{$T \colon S^{[5]} \dashrightarrow M$}{}}\label{section T}
In this section, we study the birational map $T \colon S^{[5]} \dashrightarrow M$ from the introduction, which is induced by the spherical twist at $\Ocal_S(-2)$. For the definition of the spherical twist, we refer to \cite[\S 8.1]{FM}.

\begin{prop}\label{birational map M S5}
The spherical twist at $\Ocal_S(-2)$ defines a birational map
\[T\colon S^{[5]} \dashrightarrow  M,\ \xi \mapsto T_{\Ocal_S(-2)}(\Ical_\xi)\otimes \Ocal_S(2)\]
which induces an isomorphism
\[ S^{[5]} \setminus \{ \xi \in S^{[5]} \mid h^0(\Ical_{\xi}(2)) \geq 2\} \xrightarrow \sim M\setminus ( \{ \Ecal \in M \mid h^0(\Ecal) \geq 1\} \cup M^0_\Delta).\]
In particular, $T$ is defined in $\xi \in S^{[5]}$ if there is a unique curve $D \in |2H|$ such that $\xi \subset D$. In this case,
\[ T(\xi) \cong \ker(\Ocal_S(2)|_D \rightarrow \Ocal_\xi).\]
\end{prop}

We want to point out that, due to $\rho(S)=1$, there are actually no proper birational automorphisms of $S^{[5]}$. Precisely, we have
\[\Aut(S^{[5]})=\Bir(S^{[5]}) = \langle \id, \iota^{[5]}\rangle,\]
where $\iota^{[5]}$ is the automorphism induced by the involution $\iota$ on $S$ \cite[Thm 1.1]{Al}.
Hence, $T$ is the only birational morphism $S^{[5]} \dashrightarrow M$, up to precomposition with $\iota^{[5]}$. Also note that the subvariety $\BN^1(S^{[5]})=\{ \xi \in S^{[5]} \mid h^0(\Ical_{\xi}(2)) \geq 2\} \subset S^{[5]}$ is left invariant under $\iota^{[5]}$.

\begin{proof}[Proof of Proposition \ref{birational map M S5}]
By definition of the spherical twist, there is an exact triangle in $D^b(S)$
\[ R\Gamma(\Ical_\xi(2)) \otimes \Ocal_S(-2) \rightarrow \Ical_\xi \rightarrow T_{\Ocal_S(-2)}(\Ical_\xi) \xrightarrow{[1]}.\]
So, $T_{\Ocal_S(-2)}(\Ical_\xi)$ is a complex in degrees $-1$ and $0$, which is concentrated in degree $0$ if and only if $h^0(\Ical_\xi(2))=1$, as $\chi(\Ical_\xi(2))= 1$. In this case, $T_{\Ocal_S(-2)}(\Ical_\xi)$ is as stated. \\
Let $\xi \in S^{[5]}$ and $s \in H^0(\Ical_\xi(2))$. We claim that if
\[ \Ecal \coloneqq T(\xi) \cong \ker (\Ocal_S(2)|_{D} \rightarrow \Ocal_\xi) \]
is unstable, then $h^0(\Ical_\xi(2))\geq 2$. Here, $D$ is the curve defined by the composition of $s$ with the inclusion $\Ical_\xi(2) \hookrightarrow \Ocal_S(2)$.\\
First, assume $D \in D\setminus \Sigma$. Then $\Ecal$ 
is a rank one sheaf on the integral curve $D_s$ and necessarily stable.\\
Next, if $D \in \Sigma\setminus \Delta$ write $D =D_1 \cup D_2$. Then $\Ecal$ is stable, if and only if
\[ \chi(\Ecal\otimes \Ocal_S(-1)|_{D_i}) < \tfrac{\chi(\Ecal)}{2} = -\tfrac{1}{2} < \chi(\Ecal|_{D_i})\ \text{for}\ i=1,2.\]
Otherwise, the inclusion $\Ecal\otimes \Ocal_S(-1)|_{D_i} \hookrightarrow \Ecal$ or the restriction $\Ecal \twoheadrightarrow \Ecal|_{D_i}$ to one component is destabilizing. Conversely, every destabilizing subbundle or surjection factors through the above.
We find
\[
\chi(\Ecal|_{D_i}) = -\lg(\Ocal_\xi \otimes \Ocal_{D_i}) + \chi(\Ocal_S(2)|_{D_i}) + \lg(\sheafTor_1^{\Ocal_D}(\Ocal_\xi,\Ocal_{D_i})) = \lg(\Ocal_\xi \otimes \Ocal_{D_{3-i}})-2,\]
where we used $\lg(\sheafTor_1^{\Ocal_D}(\Ocal_{D_i},\Ocal_\xi)) = 5- \lg(\Ocal_\xi \otimes \Ocal_{D_1}) - \lg(\Ocal_\xi \otimes \Ocal_{D_2})$. Similarly,
\[\chi(\Ecal\otimes \Ocal_S(-1)|_{D_i}) = \lg(\Ocal_\xi \otimes \Ocal_{D_{2-i}})-4.\] 
Hence, $\Ecal$ is unstable if and only if $\lg(\Ocal_\xi \otimes \Ocal_{D_i})\geq 4$ for one $i=1,2$.
Without loss of generality assume that $\lg(\Ocal_\xi \otimes \Ocal_{D_1})\geq 4$. There are two cases. Either there is a reduced point $x \in \xi$ such that $\xi \setminus \{x\} \subset D_1$. Otherwise, $\xi_{\rm red} \subset D_1$ and there is a point $x \in \xi$ whose multiplicity drops by one, when restricting to $D_1$. In both cases, $D_2$ can move in the pencil of curves in $|H|$ passing through $x$ and thus $h^0(\Ical_\xi(2)) \geq 2$, cf.\ Lemma \ref{Hilfslemma} below.\\
Finally, if $D = 2C \in \Delta$, the above arguments remain valid with $D_1 = D_2 =C$.
This is, if $\Ecal$ is unstable, then either $\lg(\Ocal_\xi \otimes \Ocal_C)=4$ and Lemma \ref{Hilfslemma} applies  or $\xi$ is completely contained in $C$. But then $\xi \subset C \cup C'$ for every curve $C' \in |H|$.\\
So far, we have proven that $T$ is well-defined for all $\xi \in S^{[5]}$ such that $h^0(\Ical_\xi(2))=1$. A  birational morphism between projective irreducible holomorphic symplectic manifolds is an isomorphism on the regular locus \cite[2.2]{Huy97}. Therefore, it is left to see that
\[ T (S^{[5]} \setminus \{ \xi \in S^{[5]} \mid h^0(\Ical_{\xi}(2)) \geq 2\}) = M\setminus ( \{ \Ecal \in M \mid h^0(\Ecal) \geq 1\} \cup M^0_\Delta).\]
 For sure, we have an inclusion from left to right, since $H^0(T(\xi))=0$, whenever $T$ is defined and over $\Delta$, the sheaf $T(\xi)$ always has rank one on the reduced curve. More precisely, let $\xi \in S^{[5]}$ such that $B(\xi) = 2C$. Then $\Ecal \coloneqq T(\xi)$ fits into an extension on $2C$
 \begin{equation*}
 0 \rightarrow \Lcal(x) \otimes \omega_C^{-1} \rightarrow \Ecal \rightarrow \Lcal \rightarrow 0,
 \end{equation*}
 where $x$ is the support of $\sheafTor_1^{\Ocal_{2C}}(\Ocal_C,\Ocal_\xi)$ and $\Lcal = \omega_C^{\otimes 2}(-\xi \cap C) \in \Pic^1(C)$.\\
The converse inclusion is clear over $B \setminus \Sigma$. 
Let $D=D_1 \cup D_2 \in \Sigma \setminus \Delta$ and $\Ecal \in f^{-1}(D)$ such that $h^0(\Ecal) =0$. We have to find $s \colon \Ecal \rightarrow \Ocal_S(2)|_{D}$. Then $\xi \coloneqq \Supp(\coker(s)) \in S^{[5]}$ and $\Ecal = T(\xi)$. Assume first that $\Lcal_i \coloneqq \Ecal|_{D_i}$ is torsionfree and without loss of generality that $\chi(\Lcal_i) = i-1$, i.e.\ for smooth $D_i$, we have $\Lcal_i \in \Pic^i(D_i)$. Now, we have an exact sequence
\begin{multline}\label{seq to compute image}
0 \rightarrow \Hom(\Lcal_1,\Ocal_S(1)|_{D_1}) \rightarrow  \Hom(\Ecal,\Ocal_S(2)|_D) \rightarrow  \Hom(\Lcal_2,\Ocal_S(2)|_{D_{2}}) \\
\rightarrow \Ext^1(\Lcal_1,\Ocal_S(1)|_{D_1}) \rightarrow \ldots .
\end{multline}
If $\hom(\Lcal_1,\Ocal_S(1)|_{D_1})=1$, then everything is clear and if $\hom(\Lcal_1,\Ocal_S(1)|_{D_1})=0$, then also $\ext^1(\Lcal_1,\Ocal_S(1)|_{D_1})=0$. Thus in this case $\Hom(\Ecal,\Ocal(2)|_D) \cong \Hom(\Lcal_2,\Ocal(2)|_{D_2}) \neq 0$. (Actually, we must have $\hom(\Ecal,\Ocal(2)|_D)=1$ because we assumed $h^0(\Ecal)=0$). Next, if $\Ecal|_{D_i}$ has torsion, then $\Ecal|_{D_i} \cong \Lcal_i \oplus \Tcal$, where $\Lcal_i$ is torsionfree with $\chi(\Lcal_i) = 0$ and $\Tcal$ is supported on the intersection $D_1 \cap D_2$ with $\lg(\Tcal)=1$. In particular, also in this case the sequence \eqref{seq to compute image} proves that $\Hom(\Ecal,\Ocal(2)|_D)\neq 0$.\\
Over $\Delta$ the argument is the same. Let $D =2C \in \Delta$ and assume that we are given $\Ecal \in M_{2C}^1 \setminus M_{2C}^0$ such that $h^0(\Ecal)=0$. Again, $\Ecal= T(\xi)$ if and only if $\Hom(\Ecal,\Ocal_S(2)|_{2C}) \neq 0$. This time, we have $\Ecal|_C = \Lcal \oplus \Ocal_x$ for some $\Lcal \in \Pic^1(C)$ and $x \in C$ and the sequence \eqref{seq to compute image} with $C=D_1 = D_2$ and $\Lcal = \Lcal_1 = \Lcal_2$ proves what we need.
\end{proof}

We will see in Proposition \ref{birational models}, how the indeterminancy of $T$ can be resolved by a sequence of blow-ups and blow-downs.

\begin{lem}\label{Hilfslemma}
Let $\xi \subset S$ be a zero-dimensional subscheme of length $n$ supported in a point $p\in S$. Assume there is an integral curve $C_1 \subset S$ such that $\lg(\Ocal_\xi \otimes \Ocal_{C_1})=n-1$. Then
\[ \xi \subset C_1 \cup C_2 \]
for every curve $C_2$ passing through $p$.
\end{lem}
\begin{proof}
We can assume that $S=\Spec A$, where $A$ is a local ring with maximal ideal $\mfrak$. Moreover, $\xi = V(I)$, and $C_1= V(f)$ for some $f \in A$. By assumption, $\lg(A/I)=n$ and $\lg(A/(I,f))=n-1$, hence $\lg((I,f)/I)=1$. We we want to show that $f\cdot \mfrak \subset I$ or equivalently $(I,f \cdot \mfrak)= I$. We have a short exact sequence of $\C$-vector spaces
\[ 0 \rightarrow (I,f\cdot \mfrak)/ I \rightarrow (I,f)/I \rightarrow (I,f)/(I,f\cdot \mfrak) \rightarrow 0, \]
where the middle term is of dimension one. Hence, $(f \cdot \mfrak,I)= I$ is true if and only if the right outer term is non-zero. Assume $(I,f)=(I,f\cdot \mfrak)$, then we can write $f = af +b$ for some $a \in \mfrak$ and $b \in I$. This implies $(1-a)f \in I$ and thus $f \in I$, which is a contradiction to our assumption.
\end{proof}

\section{Brill--Noether loci in \texorpdfstring{$M$}{} and \texorpdfstring{$S^{[5]}$}{}}\label{section BN}
In Proposition \ref{birational map M S5}, we established the isomorphism
\[ T \colon S^{[5]} \setminus \{ \xi \in S^{[5]} \mid h^0(\Ical_{\xi}(2)) \geq 2\} \xrightarrow \sim M\setminus ( \{ \Ecal \in M \mid h^0(\Ecal) \geq 1\} \cup M^0_\Delta).\]
In this section, we undertake a hands-on analysis of certain components of the Brill--Noether loci appearing here. Namely,
\[ \BN^i(S^{[5]}) \coloneqq \{ \xi \in S^{[5]} \mid h^0(\Ical_{\xi}(2)) \geq i+1\} \subset S^{[5]}\]
and
\[ \BN^i(M) \coloneqq \{ \Ecal \in M \mid h^0(\Ecal) \geq i+1 \} \subset M\]
for $i\geq 0$. The first is also an actual Brill--Noether locus after the identification
\[ S^{[5]} \cong M_H(1,2,0),\ \Ical_\xi \mapsto \Ical_\xi(2).\]
All these Brill--Noether loci generically have the structure of a projective bundle, which we explicitly state for certain components in Propositions \ref{W is jumping locus} and \ref{Z1}.

\subsection{Brill--Noether loci in \texorpdfstring{$M_H(1,2,0) \cong S^{[5]}$}{}}
We study the Brill--Noether locus in $S^{[5]}$ or rather $M_H(1,2,0)$ first. Our first result shows that the only non-trivial cases are $i=1,2$. For $\xi \in S^{[5]}$, we introduce the linear subspace
\[ B(\xi) \coloneqq \Pbb(H^0(S,\Ical_\xi(2)))  = \{ D \in |2H| \mid \xi \subset D\} \subset |2H|. \]

\begin{lem}\label{dim lemma}\begin{enumerate}[\rm (i)]
\item We have the inequalities
\[0 \leq h^0(S,\Ical_\xi(1)) \leq 1 \leq h^0(S,\Ical_\xi(2)) \leq 3.\]
\item If $h^0(S,\Ical_\xi(1)) = 1$, then $h^0(\Ical_\xi(2)) = 3$ and
\[ B(\xi) = m(C \times |H|) \subset \Sigma \subset |2H|, \]
where $C \in |H|$ is the unique curve containing $\xi$.
\end{enumerate}
\end{lem}
\begin{proof}
From the short exact sequence
\[ 0 \rightarrow \Ical_\xi(2) \rightarrow \Ocal_S(2) \rightarrow \Ocal_\xi \rightarrow 0, \]
it follows that $h^0(S,\Ical_\xi(2)) \geq 1$ for all $\xi \in S^{[5]}$.\\
First, assume that $B(\xi) \subset \Sigma$ then $\dim B(\xi) \leq 2$, as there is no three-dimensional linear subspace of $\Pbb^5$ that is contained in $\Sigma = \Sym^2\Pbb^2$. So, in order to show $h^0(S,\Ical_\xi(2)) \leq 3$, we can assume that $B(\xi) \cap B\setminus \Sigma \neq \emptyset$, and we can even assume that there is a smooth curve $D \in |2H|$ such that $\xi \subset D$.
We have compatible long exact sequences
\[
\xymatrix{
0 \ar[r] & 0=H^0(\Ical_\xi) \ar@{^(->}[d] \ar[r] & H^0(\Ical_\xi(2)) \ar@{^(->}[d] \ar@{^(->}[r] \ar[dr]^\alpha & H^0({\Ical_\xi}(2)|_D) \ar[d] \ar[r] & \ldots \\
0 \ar[r] & H^0(\Ocal_S)  \ar[r] & H^0(\Ocal_S(2)) \ar[r]  & H^0(\omega_D) \ar[r] & 0 ,
}
\]
where we inserted $\Ocal_S(2)|_D \cong \omega_D$. Moreover,
\[\Ical_\xi(2)|_D \cong \omega_D(-\xi)\oplus \Ocal_\xi.\]
Thus $H^0(\Ical_\xi(2)|_D)= H^0(\Ocal_\xi) \oplus H^0(\omega_D(-\xi))$ and the first summand is the kernel of the third vertical map. Hence, $\dim\im(\alpha) \leq h^0(D,\omega_D(-\xi))$.  Together, this gives
\[ h^0(\Ical_\xi(2)) \leq \dim \im(\alpha) + h^0(\Ocal_S) \leq h^0(D,\omega_D(-\xi)) +1 = h^0(D,\Ocal_D(\xi)) \leq 3, \]
where the last inequality uses Clifford's theorem \cite[IV Thm 5.4]{Hart}.\\
Next, assume that $\xi \subset C$ for a curve $C \in |H|$. Then the analogous considerations yield
\[ h^0(\Ical_\xi(1)) \leq \dim \ker(H^0(\Ocal_S(1))\rightarrow H^0(\Ocal_S(1)|_C)) + h^0(\Ical_\xi(1)|_C) - 5 = 1+ 5-5 = 1.\]
This finishes the proof of (i).\\

Next, we prove (ii). Any non-zero section $s \in H^0(\Ical_\xi(1))$ induces a short exact sequence
\[  0 \rightarrow \Ocal_S(1) \xrightarrow{s} \Ical_\xi(2) \rightarrow \ker (\Ocal_S(2)|_C \rightarrow \Ocal_\xi) \rightarrow 0,\]
which gives the isomorphism $H^0(\Ocal_S(1)) \cong H^0(\Ical_\xi(2))$ that translates into the statement for $B(\xi)$.
\end{proof}

Next, we have two strategies to find explicit components of $\BN^i(S^{[5]})$. The first relies on the observation that, given a curve $D \in |2H|$ and $x \in D$, then also $\iota(x) \in D$, where $\iota \colon S \rightarrow S$ is the covering involution of $\pi \colon S \rightarrow |H| \cong \Pbb^2$. Hence, the subvarieties in $S^{[5]}$ which parameterize subschemes that are partly invariant under $\iota$ are candidates to provide a component of the Brill--Noether locus. The second is based on Lemma \ref{dim lemma}(ii). Namely, we parameterize subschemes that are already or almost contained in a curve of the primitive linear system $|H|$.

\begin{example}\label{P}
As mentioned in the introduction, we have an embedding
\[\Pbb^2 \subset S^{[2]},\ x \mapsto \pi^{-1}(x).\]
We get generically injective rational maps
\[ \begin{array}{lcr}
g_3 \colon \Pbb^2 \times S^{[3]} \dashrightarrow S^{[5]} & \text{and} & g_1 \colon \Pbb^2 \times \Pbb^2 \times S \dashrightarrow S^{[5]}
\end{array}\]
and set
\[ P_i \coloneqq \overline{\im(g_i)} \subset S^{[5]}\ \text{for}\ i=1,3.\]
Clearly, $P_3 \subset \BN^1(S^{[5]})$ and $P_1 \subset \BN^2(S^{[5]})$. Moreover, we note $\codim P_i =5-i$ and $P_i$ is generically a $\Pbb^{5-i}$-bundle over $S^{[i]}$.
\end{example}

\begin{example}
We define
\begin{align*}
W_2 \coloneqq \{ \xi \in S^{[5]}\mid H^0(\Ical_\xi(1)) \neq 0 \} = \{ \xi \in S^{[5]}\mid \text{there is}\ C \in |H|\ \text{such that}\ \xi \subset C\}.
\end{align*}
Then $W_2$ is the closure of the image of the generically injective rational map
\[ \Sym^5_{\Ccal_{|H|}/|H|}(\Ccal_{|H|}) \dashrightarrow S^{[5]}\]
and therefore $\dim W_2 = 7$. Here, $\Ccal_{|H|} \rightarrow |H|$ is the universal curve.
We also define
\[ W_3 \coloneqq \{ \xi \in S^{[5]}\mid \text{there is}\ x \in \xi\ \text{and}\ C \in |H|\ \text{such that}\ \xi\setminus \{x\} \subset C\},\]
i.e.\ $W_3$ is the closure of the image of the generically injective rational map
\[ \Sym^4_{\Ccal_{|H|}/|H|}(\Ccal_{|H|})\times S \dashrightarrow S^{[5]}.\]
We conclude that $\dim W_3 = 8$.\\

Clearly, $W_2 \subset W_3$. By Lemma \ref{dim lemma}(i), $W_2 \subset \BN^2(S^{[5]})$. Similarly, one sees $W_3 \subset \BN^1(S^{[5]})$. Namely, if $\xi \setminus \{x\} \subset C$ as in the definition of $W_3$, then $\xi \subset C \cup C'$ for every $C' \in |H|$ such that  $x \in C'$ and thus $\dim B(\xi) \geq 1$.
\end{example}

The subvarieties $W_2$ and $W_3$ also appear in \cite[Thm 6.4]{KLM} as examples of algebraically coisotropic subvarieties in $S^{[5]}$. We can give the precise structure of a projective bundle.

\begin{prop}\label{W is jumping locus}
\begin{enumerate}[\rm (i)]
\item
The subvariety $W_2$ is a $\Pbb^3$-bundle over $M_H(0,1,-6)$. More precisely, let $\Ecal^{-6}_{\rm univ}$ be the universal bundle on $M_H(0,1,-6)\times S$ and define the sheaf
\[\Ecal_2 \coloneqq {p_1}_*R\sheafHom(\Ecal^{-6}_{\rm univ},p_2^*\Ocal_S(-1)).\]
on $M_H(0,1,-6)$. Then $\Ecal_2$ is a vector bundle and
\[ W_2 \cong \Pbb(\Ecal_2).\]
In particular, $W_2$ is smooth.\\
\item
The subvariety $W_3\setminus W_2$ is a $\Pbb^2$-bundle over an open subset of $S \times M_H(0,1,-5)$. More precisely, let $\Ecal^{-5}_{\rm univ}$ be the universal sheaf on $M_H(0,1,-5) \times S$ and $\Ical_\Delta$ the ideal sheaf of the diagonal $\Delta \subset S \times S$ and define the sheaf
\[ \Ecal_3 \coloneqq {p_{12}}_*R\sheafHom_{}(p_{23}^*\Ecal^{-5}_{\rm univ},p_3^*\Ocal_S(-1)\otimes p_{13}^*\Ical_\Delta)\]
on $S \times M_H(0,1,-5)$. 
Then $\Ecal_3$ a vector bundle on an open set $U \subset S \times M_H(0,1,-5)$
and
\[  W_3\setminus W_2 \cong \Pbb(\Ecal_3|_U) .\]
\end{enumerate}
\end{prop}

\begin{proof}
(i) For every $\Ecal \in M_H(0,1,-6)$ we have the base change map
\[ \Ecal_2(\Ecal) \rightarrow H^0(S,R\sheafHom(\Ecal,\Ocal_S(-1))) \cong \Ext^1_S(\Ecal,\Ocal_S(-1)) \]
and $\Ext^i_S(\Ecal,\Ocal_S(-1)) = 0$ for $i\neq 1$ because $\Ecal$ is stable of rank $0$. Hence $\dim\Ext^1(\Ecal,\Ocal_S(-1)) = 4$ for all $\Ecal \in M_H(0,1,-6)$ and $\Pbb(\Ecal_2)\rightarrow M_H(0,1,-6)$ is indeed a $\Pbb^3$-bundle parameterizing extensions of $\Ecal \in M_H(0,1,-6)$ by $\Ocal_S(-1)$. Moreover, as $\sheafExt^1_S(\Ecal,\Ocal_S(-1))$ is torsion free, any non-split extension
\begin{equation}\label{eq temp}
0 \rightarrow \Ocal_S(-1) \rightarrow \Ical \rightarrow \Ecal \rightarrow 0
\end{equation}
does not admit a local splitting. Hence, the middle term $\Ical$ of such an extension is torsion free \cite[Prop 1.1.10]{HL} and must be the ideal sheaf of a zero-dimensional subscheme \cite[Expl 1.1.16]{HL}. In particular, $\Ical$ is $H$-Gieseker stable and the universal extension on $\Pbb(\Ecal_2) \times S$ defines a map
\[ \psi_2 \colon \Pbb(\Ecal_2) \longrightarrow M_H(1,0,-4) = S^{[5]} \] 
whose image is $W_2$. It is left to show that $\psi_2$ is an isomorphism onto its image. For injectivity, assume that there is $\xi \in S^{[5]}$ such that $\Ical_\xi$ fits into two different extensions of the form \eqref{eq temp}. But then $h^0(S,\Ical_\xi(1)) \geq 2$ which is absurd (cf.\ Lemma \ref{dim lemma}). Finally, we have $\Hom(\Ocal_S(-1),\Ecal) =0 $, which signifies that the extensions of the form \eqref{eq temp} are rigid \cite[Thm 6.4.5]{FGA} and therefore $\psi_2$ is really an immersion of schemes.\\

(ii) In the case of $W_3$, we find
\[ \Ecal_3(x,\Ecal) \rightarrow H^0(S,R\sheafHom(\Ecal,\Ical_x(-1)))\]
and the right hand side is isomorphic to $\Ext^1_S(\Ecal,\Ical_x(-1))$ if $x \neq \Supp(\Ecal)$ and isomorphic to  $\Ext^1_S(\Ecal,\Ical_x(-1)) \oplus \C$ if $x\in \Supp(\Ecal)$. As above, we have $\dim\Ext^1(\Ecal,\Ical_x(-1))= 3$ for all $(x,\Ecal) \in S\times M_H(0,1,-5)$. Hence, $\Ecal_3$ is a vector bundle on the open subset $U \subset S \times M$ that is the inverse image of the complement of the universal curve $\Ccal_{|H|} \subset S \times |H|$ under the product of the support morphism and the identity.
Again, one checks that for $(x,\Ecal) \in U$ the middle term of every non-split extension
\begin{equation}
0 \rightarrow \Ical_x(-1) \rightarrow \Ical \rightarrow \Ecal \rightarrow 0,
\end{equation}
is a pure, $H$-Gieseker stable sheaf with Mukai vector $(1,0,-4)$ and so the associated universal extension defines a map
\[ \psi_3 \colon \Pbb(\Ecal_3|_U) \rightarrow S^{[5]} \]
whose image is clearly contained in $W_3$. We claim, that $\psi_3$ is an isomorphism onto $W_3 \setminus W_2$. Again, $\Hom(\Ical_x(-1), \Ecal) = 0$ and so $\psi_3$ is a local isomorphism. If $\Ical_\xi$ can be written in two different extensions, say over $(x,\Ecal)$ and $(x',\Ecal')$, it follows that  $\xi \setminus \{x,x'\} \subset \Supp(\Ecal)$ as well as $\xi \setminus \{x,x'\} \subset \Supp(\Ecal')$. However, a scheme of length $3$ is at most contained in one curve $C \in |H|$. Hence $\Supp(\Ecal) = \Supp(\Ecal')$, which implies $x=x'$ and also identifies the first arrow up to a scalar. In other words, all the data match and $\psi_3$ is injective. Finally, $\xi \in \im(\psi_3)$ if and only if there is $x \in \xi$ such that $\xi \setminus \{x\}$ is contained in a unique curve $C \in |H|$ but $x \notin C$. Hence, $\im(\psi_3) =W_3\setminus W_2$.
\end{proof}

In Proposition \ref{birational models}, we encounter the flop at $W_2$ and $W_3$, respectively.

\subsection{Brill--Noether loci in \texorpdfstring{$M$}{}}
In this section, we study the Brill--Noether loci in $M$.
Recall that over a smooth curve $D \in |2H|$, the fiber $f^{-1}(D)$ is isomorphic to $\Pic^3(D)$ and therefore
\[ \BN^i(M) \cap f^{-1}(D) = W^i_3(D) \subset \Pic^3(D)\]
is the classical Brill--Noether locus $W^1_3(D)$ \cite{ACGH}. It is known, that a general curve in a primitive linear system on a general K3 surface is Brill--Noether general \cite{Laz}. This implies in particular that $W_d^r$ (when non-empty) has the expected dimension
\[ \dim W^r_d = \rho(g,r,d) = g - (r+1)(g-d+1).\]
However, for non-primitive linear systems unexpected things may happen, as we encounter below.\\

First, we deal with the structure of $\BN^i(M)$ over the locus of smooth curves $B^\circ \subset B$. Our construction uses the fact, that the moduli spaces $M_H(0,2,k)$ for odd $k$ are all isomorphic (cf.\ Lemma \ref{iso mod 2}). Let $\Ccal^\circ \rightarrow B^\circ$ be the corresponding universal curve. For any $k$, we have an isomorphism
\[ M_H(0,2,k-4)^\circ  \cong \Pic^k_{\Ccal^\circ / B^\circ},\]
where $M_H(0,2,k-4)^\circ$ is the preimage of $B^\circ$ under the support map $M_H(0,2,k-4) \rightarrow B$. We define
\[ \BN^i_k(B^\circ) \coloneqq \{ \Lcal \in M_H(0,2,k-4)^\circ\mid h^0(S,\Lcal) \geq i+1 \} \subset M_H(0,2,k-4)^\circ \]
and consider its closure in two particular cases
\[ \begin{array}{lcr}
Z_1 \coloneqq \overline{\BN^0_1(B^\circ)} \subset M_H(0,2,-3) & \text{and} & Z_3 \coloneqq \overline{\BN^0_3(B^\circ)} \subset M = M_H(0,2-1).
\end{array}\]
We expect $Z_3 \subset BN^0(M)$ to be a strict inclusion, as the latter might have components over $\Sigma$ or $\Delta$.

In the following, we will consider $Z_1$ as a subvariety of $M$ via the isomorphism
 \[ M_H(0,2,-3) \xrightarrow\sim M,\ \Ecal \mapsto \sheafExt^1(\Ecal,\Ocal_S)(-1).\]
In particular, over a smooth curve $D \in |2H|$, we have
\[ \Pic^1(D) \rightarrow \Pic^3(D),\ \Lcal \mapsto \Lcal^\vee \otimes \Ocal_S(1)|_D\]
and
\[ Z_1 \cap f^{-1}(D) = \{ \Lcal \in \Pic^3(D) \mid h^0(\Lcal \otimes \Ocal_S(1)|_D) \geq 4\} = \{ \Lcal \in \Pic^3(D) \mid h^1(\Lcal \otimes \Ocal_S(1)|_D) \neq 0\}.\]

\begin{lem}
We have
\[ Z_1 \subset \BN^1(M). \]
In particular, there is an inclusion
\[ Z_1 \subset Z_3. \]
\end{lem}

In Corollary \ref{Z1 ohne Z3}, we prove that actually $Z_3 \cap \BN^1(M) = Z_1$.

\begin{proof}
It suffices to show the result over a smooth curve $D \in |2H|$. Let $\Lcal \in \Pic^1(D)$ such that $H^0(D,\Lcal) \neq 0$. We want to show that $h^0(D,\Lcal^\vee \otimes \Ocal_S(1)|_D) \geq 2$. Write $\Lcal = \Ocal_D(x)$ for a point $x \in D$. On $S$, we have a short exact sequence
\begin{equation*}
0 \rightarrow \Ocal_S(-1) \rightarrow \Ical_x(1) \rightarrow \Ocal_D(-x) \otimes \Ocal_S(1)|_D \rightarrow 0
\end{equation*}
and the resulting long exact cohomology sequence proves the lemma.
\end{proof}

\begin{prop}\label{Z1} 
\begin{enumerate}[\rm (i)]
\item There is an embedding
\[ \xymatrix{ \Ccal \ar@{^(->}[r] \ar[dr] & M \ar[d]^f \\
& B,
}\]
whose image is $Z_1$. In particular, $\dim Z_1 = 6$ and $Z_1$ is a $\Pbb^4$-bundle over $S$.
\item $Z_3$ is generically isomorphic to a $\Pbb^2$-bundle over $S^{[3]}$. In particular, $\dim Z_3 = 8$.
\end{enumerate}
\end{prop}


\begin{proof}
The proof of (i) and (ii) works analogously. The idea is that, over $D \in B^\circ$ we want to parameterize the line bundles $\Ocal_D(-\xi_i), i=1,3$ for a point $\xi_1 \in D$ and a divisor $\xi_3 \subset D$ of degree 3, respectively. These are the ideal schemes of $\xi_i \subset D$ and can be realized as the quotient of the respective ideal sheaves in $S$
\[ 0 \rightarrow \Ical_{D / S} = \Ocal_S(-2) \rightarrow \Ical_{\xi_i / S} \rightarrow  \Ical_{\xi_i / D} = \Ocal_D(-\xi_i) \rightarrow 0. \]
Hence, our task is to parameterize $\xi_i \subset D$ for $\xi_i \in S^{[i]}$ and $D \in B$, define the above sequence universally and show that the cokernel defines a map to $M_H(0,2,-i-4)$ for $i=1,3$.

The first step is straightforward. For $i=1,3$ define
\[ \Xcal_i  \coloneqq \Pbb({p_2}_*(\Ical_{\mathcal{Z}_i}\otimes p_1^*\Ocal_S(2))) \rightarrow S^{[i]},\]
where ${\mathcal{Z}_i} \subset S\times S^ {[i]}$ is the universal subscheme and $p_i$ are the projections from $S \times S^{[i]}$ for $i=1,2$. The inclusion
\[ {p_2}_*(\Ical_{\mathcal{Z}_i}\otimes p_1^*\Ocal_S(2)) \hookrightarrow {p_2}_*(\Ocal_{S \times S^{[i]}}\otimes p_1^*\Ocal_S(2)) \cong H^0(S,\Ocal_S(2))\otimes \Ocal_{S^{[i]}}\]
defines an embedding $\Xcal_i \subset S^{[i]} \times B$. We could also think of $\Xcal_i$ as $\Hilb^i(\Ccal/B)$, i.e.
\[ \begin{array}{lcr}
\Xcal_1 = \Ccal = \{ p \in D \} \subset S \times B& \text{and} &\Xcal_3 = \{ \xi \subset D\} \subset S^{[3]} \times B.
\end{array}\]
Note that $\Xcal_1$ is a $\Pbb^4$-bundle over $S$ and $\Xcal_3$ is generically a $\Pbb^2$-bundle over $S^{[3]}$. On $S \times \Xcal_i$, we have the sequence
\[ 0 \rightarrow (\id \times p_B)^*\Ocal_{S\times B}(-\Ccal) \rightarrow (p_S\times \id)^*\Ical_{\mathcal{Z}_i} \rightarrow \Qcal_i \rightarrow 0. \]
Here, $\Qcal_i$ is defined to be the cokernel, which is flat over $\Xcal_i$ and $v(\Qcal_i|_{S \times \{p\}} ) = (0,2,-i-4)$ for all $p \in \Xcal_i$. Consequently, $\Qcal_i$ gives a map
\[ \xymatrix{ \Xcal_i \ar[dr] \ar@{-->}[rr]^-{\varphi_i} && M_H(0,2,-i-4), \ar[dl]& p \mapsto \Qcal_i|_{S \times \{p\}} \\
& B } \]
which is defined in $p\in \Xcal_i$, whenever $\Qcal_i|_{S \times \{p\}}$ is stable. By definition, we have $\im(\varphi_i) \subset Z_i$.\\
For simplicity, we restrict to the case $i=1$ in the rest of the proof. Actually, the more powerful methods from Proposition \ref{birational models} allow us to conclude without explicit computation, that $Z_3\setminus Z_1$ is a projective bundle.\\
We claim that $\varphi_1$ is everywhere defined and immersive. Clearly, $\varphi_1$ is defined and injective over $B\setminus \Sigma$ and
with the same arguments as in the proof of Proposition \ref{birational map M S5} this also holds true over $\Delta$. So $\im(\varphi_1)=Z_1$ and 
it is left to show that $\varphi_1$ is an immersion. We show that the induced map on tangent spaces is injective. To this end, let $S[\epsilon] \coloneqq S \times_\C \C[\epsilon] \xrightarrow p S$ and assume we are given a $\C[\epsilon]$-valued point of $\Ccal$ that maps to a trivial deformation of $\Ecal \in M_H(0,2,-5)$, i.e.\ this point corresponds to a sequence
\begin{equation}\label{defo}
p^*\Ocal_S(-2) \rightarrow \Scal \rightarrow p^*\Ecal \rightarrow 0
\end{equation}
on $S[\epsilon]$. Here, we use that the line bundle $\Ocal_S(-2)$ is rigid. We want to see that $\Scal = p^*\Ical_x$, where $\Ecal$, as before, sits in the sequence $0 \rightarrow \Ocal_S(-2) \rightarrow \Ical_x \rightarrow \Ecal \rightarrow 0$. By definition, \eqref{defo} embeds into a diagram
\[
\xymatrix{
p^*\Ocal_S(-2) \ar[r]\ar@{=}[d]& \Scal \ar[r]\ar[d] & p^*\Ecal \ar[r]\ar[d] & 0 \\
p^*\Ocal_S(-2) \ar[r]          & \Ocal_{S[\epsilon]} \ar[d]\ar[r] & \Ocal_{\tilde{D}}\ar[d] \ar[r]& 0\\
& \Ocal_{\tilde{x}} \ar@{=}[r] \ar[d] & \Ocal_{\tilde{x}} \ar[d] \\
& 0 & 0,
}
 \]
where $\tilde{x} \subset \tilde{D} \subset S[\epsilon]$ are deformations of $x$ and $D$. As $\Supp(p^*\Ecal)= D[\epsilon]$, we must have $p^*\Ecal = \Ocal_{D[\epsilon]}(-x[\epsilon])$ and we can conclude that all deformations are trivial. Hence, $\Scal = p^*\Ical_x$.
\end{proof}

Note that the smooth curves in $|2H|$ are hyperelliptic and so there is a unique line bundle $g^1_2(D) \in \Pic^2(D)$ such that $h^0(g^1_2)=2$.

\begin{cor}
Let $D \in |2H|$ is a smooth curve and $\Lcal \in f^{-1}(D)$. Then
\[ \Lcal \in  Z_1 \cap f^{-1}(D) \iff  \Lcal \cong \Ocal_D(x) \otimes g^1_2\ \text{for some}\ x \in D.\]
\end{cor}
\begin{proof}
By dimension reasons, it suffices to show one implication. Assume $\Lcal \cong \Ocal_D(x) \otimes g^1_2$ for some $x \in D$. We know that $\Ocal_S(1)|_D \cong (g^1_2)^{\otimes 2}$ \cite[Prop 5.3]{Hart}. Thus
\[h^1(\Lcal \otimes \Ocal_S(1)|_D) = h^1(\Ocal_D(x) \otimes (g^1_2)^{\otimes 3}) = h^0(\Ocal_D(-x) \otimes g^1_2) \neq 0,\]
which proves the claim.
\end{proof}

\begin{rem}\label{remark bn}
Another consequence of Proposition \ref{Z1} is that $W^0_3(D)$ has the expected dimension
 \[\dim W^0_3(D)  = \rho(5,0,3) =3 \]
 and that
\[\dim W^1_3(D) \geq 1\]
for $D \in |2H|$ general, even though the Brill--Noether number $\rho(5,1,3) = 5-2(5-3+1)$ is negative. The latter also follows, since $g^1_2(D) \otimes \Ocal_D(x) \in W^1_3(D)$ for all $x \in D$. In Corollary \ref{Z1 ohne Z3}, we will see that every line bundle in $W^1_3(D)$ is of this form, i.e.\ $\dim W^1_3(D) =1$.  Finally, we know
\[W^2_3(D) = \emptyset\]
from Clifford's theorem \cite[IV Thm 5.4]{Hart}.
\end{rem}

\section{Computation of the birational models}\label{section models}
In this final section, we leave our hands-on methods behind and apply the techniques of Bayer and Macr\`i to get a full picture of the birational models of $S^{[5]}$ and the associated birational wall-crossing transformations. We find that there are five birational models of $M$ (including $M$ and $S^{[5]}$) and moreover match the exceptional loci of the flopping contractions with the subvarieties from the previous section. By a birational model of $M$, we mean a smooth projective variety with trivial canonical bundle that is birational to $M$.

\subsection{Numerical characterization of the walls in \texorpdfstring{$\Mov(M)$}{}} \label{section num walls}
We will compute the wall and chamber decomposition of the movable cone of $S^{[5]}$ (resp.\ $M$), whose chambers correspond to the birational models of $S^{[5]}$ (resp.\ $M$) using Bayer and Macr\`i's results \cite{MMP}. We start by recalling the basic definitions and relevant statements in this context.\\

Let $X$ be an irreducible holomorphic symplectic manifold. Recall that the \emph{positive cone} $\Pos(X) \subset \NS(X)_\R$ is the connected component of $\{x \in \NS(X)_\R \mid (x,x) > 0 \}$ containing a K\"ahler class. The \emph{movable cone} $\Mov(X) \subset \NS(X)_\R$ is the open cone generated by the classes of divisors $D$ such that $|D|$ has no divisorial base locus.  We have the inclusions
\[ \overline{\Amp(X)}= \Nef(X) \subset \overline{\Mov(X)} \subset \overline{\Pos(X)} \subset \NS(X)_\R. \]
The movable cone admits a locally polyhedral chamber decomposition, whose chambers correspond to smooth birational models of $X$. More precisely,
\[ \overline{\Mov(X)} = \overline{\bigcup_g g^*\Nef(X')}, \]
where the union is taken over all birational maps $g\colon X \dashrightarrow X'$ from $X$ to another irreducible holomorphic symplectic manifold $X'$ \cite[Thm 7, Cor 19]{HT}.\\

Assume that $X=M_\sigma(v)$ is a smooth projective moduli space of $\sigma$-stable objects in $D^b(S)$ with $v^2>0$. In this case, the Mukai morphism \cite[Thm 3.6]{MMP} gives the identification
\[ \lambda_X \colon v^\perp \xrightarrow\sim \NS(X).\]
Here, $v^\perp \subset H^*_{\rm alg}(S,\Z)$. By \cite[Thm 12.1]{MMP}, the nef (resp. movable) cone of $X$ is one of the chambers of the decomposition of the positive cone $\overline{\Pos(X)}$ whose walls are the orthogonal complement to linear subspaces
\[  \lambda_{X}(\mathcal{H}^\perp),\]
where $\mathcal{H} \subset H^*_{\rm alg}(S,\Z)$ is a primitive sublattice of signature $(1,1)$ that contains $v$. If $\mathcal{H}= \langle v,a \rangle$, then $a \in H^*_{\rm alg}(S,\Z)$ can be chosen such that $a^2 \geq -2$ and $0 \leq (v,a) \leq \tfrac{v^2}{2}$ (resp.\ ($a^2=-2$ and $(v,a)=0$) or ($a^2=0$ and $(a,v) \in \{1,2\})).$
Following \cite[Thm 5.7]{MMP}, the lattice $\mathcal{H}$ governs the geometry of the birational transformation at the respective wall. One can distinguish the following cases:
\begin{enumerate}[\rm (a)]
\item The lattice $\mathcal{H}$ is isomorphic to one of the following:
$\left(\begin{smallmatrix}-2 & 0 \\ 0& v^2\end{smallmatrix}\right), \left(\begin{smallmatrix}0 & 1 \\ 1 & v^2\end{smallmatrix}\right), \left(\begin{smallmatrix}0 & 2 \\ 2& v^2\end{smallmatrix}\right) $. This case corresponds to a divisorial contraction.
\item The lattice $\mathcal{H}$ is none of the above and there is either $s \in \mathcal{H}$ such that $s^2=-2$ and $0 < (s,v) \leq \tfrac{v^2}{2}$ or $v$ is the sum $v = a_1 + a_2$ of two positive classes $a_i \in \mathcal{H}$ (i.e.\ $a_i^2 \geq 0$ and $(a_i,v)> 0$ for $i=1,2$). 
This case corresponds to a flopping contraction.
\item  In all other cases, the birational transformation is actually an isomorphism. 
\end{enumerate}

The rough idea here is, that a wall of the ample cone is induced by a wall-crossing in the space of stability conditions and the associated contraction contracts precisely the curves of objects that are S-equivalent with respect to the stability condition on the wall \cite[Thm 1.4(a)]{BM}. To a wall $\mathcal{W}$ (with respect to $v$) of the stability manifold, Bayer and Macr\`i associate a rank two sublattice \cite[Prop 5.1]{MMP}
\[ \mathcal{H} \coloneqq \{ w \in H^*_{\rm alg}(S,\Z) \mid \phi_0(w) = \phi_0(v)\ \text{for all}\ \sigma_0 \in \Wcal\} \subset H^*_{\rm alg}(S,\Z).\]
Here, $\phi_0$ is the phase associated to $\sigma_0=(Z_0,\mathcal{A}_0)$. Then $\Hcal$ has the property that if $\Ecal$ is a $\sigma$-stable object and $A_i$ is a factor in its Harder--Narasimhan filtration with respect to a stability condition $\sigma_-$, which lies sufficiently close on the other side of the wall, then $v(A_i) \in \Hcal$. Now, let $\Ecal_1$ and $\Ecal_2 \in M_\sigma(v)$ have the same Harder--Narasimhan factors with respect to $\sigma_-$. As one can always find a Jordan--H\"older filtration that is a refinement of the Harder-Narasimhan filtration, this implies that $\Ecal_1$ and $\Ecal_2$ are $S$-equivalent and therefore contracted under the transformation induced from the wall-crossing. Consequently, in order to understand this transformation, one has to parameterize possible Harder--Narasimhan filtrations whose factors have Mukai vectors in $\Hcal$. Unfortunately, it may also happen that the open subset $M_{\sigma_0}^{\rm st}(v) \subset M_{\sigma_0}(v)$ of stable objects is empty. In this case, the behavior at the wall is harder to control and we call $\Wcal$ a \emph{totally semistable wall}. By \cite[Thm 5.7]{MMP}, $\mathcal{W}$ is totally semistable if and only if
\begin{enumerate}[\rm (a')]
\item there is $w \in \Hcal$ such that $w^2=0$ and $(v,w)=1$ or
\item there is $s \in \Hcal$ such that $s^2=-2, M_{\sigma_0}(s)\neq \emptyset$ and $(s,v)<0$.
\end{enumerate}
The content of Bayer and Macr\`i's article \cite{MMP} is a detailed study of the possible lattices and the associated modifications of the moduli space, which in particular, yields the above lists.

For our computations, we fix the following notation:
We set \[ \begin{array}{lcr}
v=(0,2,-1) & \text{and} & v'= (1,0,-4).
\end{array} \]
The respective Mukai morphisms
fit in the commutative diagram
\begin{equation*}
\xymatrix{
H^*_{\rm alg}(S,\Z)\ar[d]^{T^*}_\cong  & v^\perp\ar[d]_\cong \ar@{_(->}[l] \ar[r]^-{\lambda_M} 	& \NS(M) \ar[d]^{T^*}_\cong\\
H^*_{\rm alg}(S,\Z)  				& v'^\perp \ar@{_(->}[l]\ar[r]^-{\lambda_{S^{[5]}}} & \NS(S^{[5]}),
}
\end{equation*}
where, by abuse of notation, we also write $T^*\colon H^\ast_{\rm alg}(S,\Z) \rightarrow H^\ast_{\rm alg}(S,\Z)$ for the isomorphism that makes the left square commute. It is defined as the composition
\[ H^\ast_{\rm alg}(S,\Z) \xrightarrow{\cdot \ch(\Ocal_S(-2))} H^\ast_{\rm alg}(S,\Z) \xrightarrow{\rho_{v(\Ocal_S(-2))}} H^\ast_{\rm alg}(S,\Z),\]
where $\rho_{v(\Ocal_S(-2))}$ is the reflection at the hyperplane orthogonal to $v(\Ocal_S(-2))= (1,-2,5)$. In our usual basis, $T^*$ is given by the matrix
\begin{equation}\label{matrix T}
\begin{pmatrix}
-4 & -4 & -1 \\ 10 & 9 & 2 \\ -25 & -20 & -4
\end{pmatrix} \circ
\begin{pmatrix}
1 & 0 & 0\\ -2 & 1 & 0 \\ 4 & -4& 1
\end{pmatrix} =
\begin{pmatrix}
0& 0 & -1\\ 0 & 1 & 2 \\ -1 & -4& -4
\end{pmatrix}.
\end{equation}

We have the following 
basis of $\NS(S^{[5]})$ 
\[
\begin{array}{lcr}
\delta \coloneqq \lambda_{S^{[5]}}(-1,0,-4) & \text{and} & H \coloneqq \lambda_{S^{[5]}}(0,-1,0).
\end{array}
\]

For the Hilbert scheme $S^{[n]}$, computing the walls in $\overline{\Pos(S^{[n]})}$ reduces to solving Pell's equation, cf. \cite[Prop. 13.1]{MMP} and also \cite[Lem 2.5]{Cat}. In our case, we get the following list of walls  with respective intersection properties:
\[ \begin{array}{c|c|c|c|c|c|c|c}
i & a'_i \in H^\ast_{\rm alg}(S,\Z) & (a'_i,a'_i) & (a'_i,v') & D_i  \in H^2(S^{[5]},\Z) & (D_i,D_i) & R_i \in H_2(S^{[5]},\Z) & (R_i,R_i) \\
\hline
0 & (0,0,1)  & 0& 1 &  -\delta      & -8   &  -\delta^\vee    & -\tfrac{1}{8}\\
1 & (1,-1,2) & -2& 2 & 4H-3\delta   & -40  & H-6\delta^\vee  & -\tfrac{5}{2}\\
2 & (1,-1,1) &  0& 3 & 8H-5\delta   & -72  & H-5\delta^\vee  & -\tfrac{9}{8}\\
3 & (-1,2,-5)& -2& 1 & -16H+9\delta  & -136 & -2H+9\delta^\vee & -\tfrac{17}{8}\\
4 & (2,-3,5) & -2& 3 & 24H-13\delta & -200 & 3H-13\delta^\vee & -\tfrac{25}{8}\\
5 & (-1,2,-4)&  0& 0 & -2H+\delta    & 0    & -H+4\delta^\vee   & 0.
\end{array}
\]
Here, $\pm D_i \in \NS(S^{[5]})$ is the integral class defining the same wall as $a'_i$, i.e.\
\[ D_i^\perp = \lambda_{S^{[5]}}({a'}_i^\perp \cap {v'}^\perp)\]
or, in other words, $D_i$ is a rational multiple of the orthogonal projection of $a'_i$ to $({v'}^\perp)_\Q$. The sign here is chosen such that this multiple is positive, but it does not matter. In \cite{Mo}, the classes $\pm D_i$ are called wall divisors. Moreover, $R_i \in H_2(S^{[5]},\Z)$ is the curve class corresponding to $D_i \in H^2(S^{[5]},\Z)$, i.e.\ $D_i$ is the smallest positive multiple of $R_i$ contained in $H^2(S^{[5]},\Z)$ under the embedding $H^2(S^{[5]},\Z) \subset H_2(S^{[5]},\Z)$ coming from the intersection form. Below, we also give the list of walls in coordinates of $M$.

\begin{cor}
For a K3 surface $S$ with $\Pic(S)=\Z\cdot H$ and $H^2=2$, there are five smooth birational models of $S^{[5]}$ (including $S^{[5]}$ itself).
\end{cor}

\subsection{Geometrical characterization of the walls in \texorpdfstring{$\Mov(M)$}{}}
So far, we know that the movable cone of $M$ or $S^{[5]}$, respectively, is divided into five chambers. The outer ones correspond to $M$, with the Lagrangian fibration and to $S^{[5]}$ with the Hilbert--Chow morphism. Next, we want to understand exceptional loci (and their strict transforms in $M$ and $S^{[5]}$, respectively) of the birational transformations between two models in adjacent chambers.\\

The geometry of the occuring contractions is studied in-depth in \cite[\S 9]{MMP}, to which we refer for the precise results. As mentioned above, the rough idea is to parameterize objects with prescribed Harder-Narasimhan filtration with respect to a stability condition on the other side of the wall. This translates into finding decompositions $v = a_1 + \ldots + a_m$ into effective classes $a_i \in \Hcal$, where $\Hcal  \subset H^*_{\rm alg}(S,\Z)$ is the sublattice such that $\lambda_{M_\sigma(v)}(\mathcal{H}^\perp)$ cuts out the wall of the ample cone. Here, a class $a \in \Hcal$ is called \emph{effective} if $M_{\sigma_0}(a) \neq \emptyset$ \cite[Prop 5.5]{MMP}, and a class $a \in H^*_{\rm alg}(S,\Z)$ is called \emph{positive}, if $a^2\geq0$ and $(a,v)> 0$. All positive classes are effective. \\

Let $\mathcal{H}$ define a flopping wall for $M_\sigma(v)$. By \cite[Prop 9.1]{MMP} there are two cases: Either
\begin{enumerate}[\rm (i)]
\item there is a decomposition $v=a + b$ into two positive classes and $\Hcal = \langle v,a \rangle$. Or
\item there is a spherical class $s \in \Hcal$ such that $0 < (s,v) \leq \tfrac{v^2}{2}$.
\end{enumerate}
In case (i), assume moreover that the wall is not totally semistable with respect to $a$ or $b$ (e.g.\ if $\Hcal$ does not contain any spherical or isotropic classes), and that $\phi_\sigma(a) < \phi_\sigma(b)$, where $\phi_\sigma$ is the phase of the stability condition $\sigma$. By \cite[\S 9]{MMP}, the decomposition $v=a+b$ defines an irreducible component $E$ of the exceptional locus of the contraction morphism associated to the wall defined by $\mathcal{H}$, such that a generic point $\Ecal \in E$ is an extension
\[ \Acal \rightarrow \Ecal \rightarrow \Bcal \xrightarrow{[1]}, \]
where $\Acal$ and $\Bcal$ are $\sigma$-stable objects of Mukai vector $a$ and $b$, respectively. By assumption on the wall, $\Acal$ and $\Bcal$ are generically also $\sigma_0$-stable, where $\sigma_0$ is a generic stability condition on the wall. And by definition of $\Hcal$, we have $\phi_{\sigma_0}(a) = \phi_{\sigma_0}(b)$. Hence, $\Hom(\Bcal,\Acal)=0= \Hom(\Acal,\Bcal)$. Finally, $\Ecal$ defines a class in $\Pbb^1(\Ext^1_{D^b(S)}(\Bcal, \Acal))$ and we find 
\[r \coloneqq \dim \Pbb^1(\Ext^1_{D^b(S)}(\Bcal,\Acal)) = (v-a,a)-1.\]
Thus, $E$ has generically the structure of a $\Pbb^r$-bundle over $M_\sigma(a)\times M_\sigma(b)$. By assumption, $M_\sigma(a)$ and $M_\sigma(b)$ are both non-empty and
\begin{align*}
\codim E &= (v,v)+2 - \dim M_\sigma(a) - \dim M_\sigma(b) -\dim \Pbb^1(\Ext^1_{D^b(S)}(\Bcal,\Acal)) \\
		&= (v,v) +2-((a,a)+2)- ((v-a,v-a)+2) - ((v-a,a) -1)\\
		&= (v-a,a)-1 =r.
\end{align*} 
It may happen that case (ii) is a special case of (i). Otherwise $\mathcal{H}= \langle v,s\rangle$ and the above results also hold for the decomposition  $v = s + (v-s)$ if $s$ is effective and need extra care, if $s$ is not effective \cite[Proof of Prop 9.1]{MMP}.\\


The relevant data for our example is listed in the following table:
\[ \begin{array}{c|c|c|c|c|c}
i & a'_i \in H^\ast_{\rm alg}(S,\Z) & a_i \in H^\ast_{\rm alg}(S,\Z) & (a_i,a_i) & (a_i,v) & r_i  \\
\hline
0 & (0,0,1)  & (-1,0,0)  & 0& 1 & \\
1 & (1,-1,2) & (-2,1,-1) &-2& 2 & 3\\
2 & (1,-1,1) & (-1,1,-1) & 0& 3 &2\\
3 & (-1,2,-5)&  (1,0,1)  &-2& 1 & 2\\
4 & (2,-3,5) & (-1,1,-2) &-2& 3 & 4\\
5 & (-1,2,-4)& (0,0,1)   & 0& 0. &
\end{array}
\]

Here, $T^*a_i = a'_i$ (cf.\ \eqref{matrix T}), so that 
\[ \lambda_M(a_i^\perp \cap v^\perp) =  \lambda_{S^{[5]}}({a'_i}^\perp \cap {v'}^\perp)\]
cuts out the walls in $\overline{\Pos(M)} \cong \overline{\Pos(S^{[5]})}$.

When following the birational transformations from $S^{[5]}$ to $M$, not all exceptional loci can naively be identified in $S^{[5]}$ as their codimension $r_i = (v-a_i,a_i) -1$ is not strictly decreasing. However, if we approach from either side, we get a beautiful description.

\begin{theo}\label{birational models}
The birational models of $S^{[5]}$ are connected by the following chain of flopping contractions
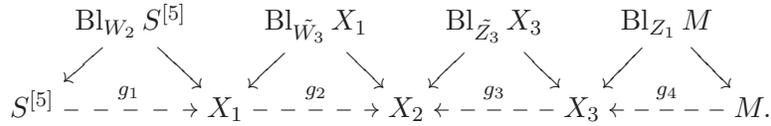
\[ \xymatrix@R-1pc@C-2pc{
& \Bl_{W_2}S^{[5]} 	   \ar[dr]\ar[dl] &&
\Bl_{\tilde{W_3}}X_1   \ar[dr]\ar[dl] &&
\Bl_{\tilde{Z_3}}X_3   \ar[dr]\ar[dl] &&
\Bl_{Z_1}M             \ar[dr]\ar[dl] \\
S^{[5]} \ar@{-->}[rr]^{g_1} && X_1 \ar@{-->}[rr]^{g_2} && X_2 && X_3 \ar@{-->}[ll]_{g_3} && M. \ar@{-->}[ll]_{g_4}} \]
Here, $\tilde{W_3}$ (resp.\ $\tilde{Z_3}$) is the strict transform of $W_3$ (resp.\ $Z_3$) under $g_1$ (resp.\ $g_4$).
\end{theo}


\begin{rem}\label{primitive case}
In the primitive case, the birational transformation $S^{[g]} \dashrightarrow M_H(0,1,1)$, where $H^2 = 2g-2$ can be resolved in one step as follows (e.g.\ \cite[\S3]{ADM})
\begin{equation*}
\xymatrix{
&  \Hilb^g(\Ccal/|H|) \ar[dl] \ar[dr] & && (\xi, D) \ar@{|->}[dl] \ar@{|->}[dr]  &\\
S^{[g]} \ar@{-->}[rr]& & M_H(0,1,1),& \xi && \Ocal_D(\xi),
}
\end{equation*}
where
\[ \Hilb^g(\Ccal/|H|) = \{(\xi,D)\mid \xi \subset D \} \subset S^{[g]} \times |H|.\]
If $g=2$, the birational map $S^{[2]}\dashrightarrow M_H(0,1,1)$ is the original example of a Mukai flop \cite[Expl 0.6]{Mu}.\\
More generally, the geometry of the birational map $S^{[g]} \dashrightarrow M_H(0,1,1)$, where $H^2 = 2g-2$ is studied in \cite{Mar}.
\end{rem}

\begin{proof}[Proof of Theorem \ref{birational models}]
We verify wall by wall that the contracted locus is as claimed in the Proposition. This means first of all, that we have to find decompositions $v = a + b$ inside $\mathcal{H}_i = \langle v,a_i \rangle$ for $i=1,\ldots,4$ corresponding to a flopping contraction. If we set $a = xa_i + yv$, solving for a positive decomposition reduces to solve
\[ \left\lbrace {\begin{array}{cc}
& a_i^2x^2 + 2(a_i,v)xy + 8y^2 \geq 0 \\
& 0 < (a_i,v)x+8y \leq 4
\end{array}} \right. \]
for integer solutions $(x,y)$ and similar for decompositions, where $a$ is a spherical class. Due to the small dimension of our example, we find by explicit computation that at every wall $v = a_i + (v-a_i)$ is the only suitable decomposition. We also compute that none of the walls is totally semistable with respect to $v$. Moreover, for all $i=1,\ldots,4$, the parallelogram inside $\Hcal_i \otimes \R$ with vertices $0,a_i,v-a_i,v$ does not contain any other lattice points and so the decomposition does not admit a refinement. This is reflected by the fact that the exceptional locus in each step is irreducible and actually a projective bundle.\\

We attack the first two walls starting from $S^{[5]}$. The wall crossing can be realized by the following  path in the stability manifold
\[ \sigma'_t \coloneqq \sigma_{tH,-2H} = (Z_{tH,-2H},\Coh^{-2}(S)),\ t\in (0,+\infty),\]
that actually hits every wall for $v'$ \cite[Thm 10.8]{BM}. (For the definitions, see \cite[Prop 7.1]{Bri} or the summary in \cite[\S 6]{BM}). Explicitly, the walls arise when $Z_{tH,-2H}(v)$ and $Z_{tH,-2H}(a_i)$ are $\R$-linearly dependent. We find
\[ +\infty > t_0 \coloneqq 2 > t_1 \coloneqq \sqrt{2} > t_2 \coloneqq \sqrt{\tfrac{5}{3}} > t_3 \coloneqq \sqrt{\tfrac{2}{3}} > 0\]
such that
\[ M_{\sigma'_t}(v') \cong \left\lbrace \begin{array}{ll}
S^{[5]} &\text{for } t> t_0\\
X_i     &\text{for } t_{i-1} > t > t_{i}\\
M     &\text{for } t_3 > t > 0.
\end{array}\right.\]
At the first wall, we have the sublattice $\Hcal_1= \langle a'_1,v'\rangle \cong (\Z^2,\left(\begin{smallmatrix}-2 & 2 \\ 2& 8\end{smallmatrix}\right))$. This lattice 
admits no decomposition of $v'$ into positive classes. But we have
\[ v'= a'_1 + b'_1\ \text {with}\ a'_1=(1,-1,2)= v(\Ocal_S(-1))\ \text{and}\ b'_1=(0,1,-6)\]
and $a'_1$ is the only spherical class $s$ with $0 < (s,v) \leq 4$ in $\Hcal_1$. Hence an ideal sheaf $\Ical_\xi \in S^{[5]}$ is in the exceptional locus of $g_1\colon S^{[5]} \dashrightarrow X_1$ if and only if it fits into a short exact sequence
\[ 0 \rightarrow \Ocal_S(-1) \rightarrow \Ical_\xi \rightarrow \Qcal \rightarrow 0\]
with $\Qcal \in M_H(0,1,-6)$. By Proposition \ref{W is jumping locus}, this is equivalent to $\xi \in W_2$. Hence $g_1$ is the flop at the projective bundle $W_2$ (see also \cite[Expl 10.2]{BM}).

The second wall corresponds to the lattice $\Hcal_2= \langle a'_2,v'\rangle \cong (\Z^2,\left(\begin{smallmatrix}0 & 3 \\ 3& 8\end{smallmatrix}\right))$, which contains no suitable spherical classes and admits exactly one decomposition into positive classes. Namely,
\[ v'=a'_2 +b'_2\ \text {with}\ a'_2 = (1,-1,1)\ \text{and}\ b'_2 = (0,1,-5).\]
Choose $t_1 < r < t_0$ such that $X_1= M_{\sigma'_r}(v')$. Let $\xi \in S^{[5]}\setminus W_2$. Then $\Ical_\xi$ is not destabilized at the first wall and hence $\Ical_\xi \in X_1$. Now, $\Ical_\xi$ is in the exceptional locus of $g_2 \colon X_1 \dashrightarrow X_2$ if and only if there is an exact triangle
\[ \Acal \rightarrow \Ical_\xi \rightarrow \Bcal \xrightarrow{[1]},\]
where $\Acal \in M_{\sigma'_r}(a_2')$ and $\Bcal \in M_{\sigma'_r}(b_2')$ are stable objects. We claim that 
 $M_{\sigma'_r}(a'_2)$ is isomorphic to the original K3 surface $S$, via $S \ni x \mapsto \Ical_x(-1)$, where $\Ical_x \subset \Ocal_S$ is the ideal sheaf of the point $x \in S$. Indeed, $\Ical_x(-1) \in \Coh^{-2}(S)$ for all $x \in S$. Moreover, $M_{\sigma'_r}(b_2')\cong M_H(0,1,-5)$, as there is only one wall for $b_2'$, which is defined by $v(\Ocal_S(-1))=(1,-1,2)$ and hit by our path for $t=\sqrt{2}= t_1$ (cf.\ Remark \ref{primitive case}). Consequently, $\Ical_\xi$ is contracted if and only if there is $x \in S, \Qcal \in M_H(0,1,-5)$ and an extension
\[ 0 \rightarrow \Ical_x(-1) \rightarrow \Ical_\xi \rightarrow \Qcal \rightarrow 0.\]
In other words, $\xi \in W_3$, see Proposition \ref{W is jumping locus}.

The remaining walls, we detect starting from $M$ along the path $\sigma_t \coloneqq \sigma_{tH,0}$ for $t \in (1,+\infty)$. Then
\[ M_{\sigma_t}(v) \cong \left\lbrace \begin{array}{ll}
M     &\text{for } t > \tfrac{\sqrt{6}}{2} < t\\
X_3 & \text{for } \tfrac{\sqrt{6}}{2} > t > 1,
\end{array}\right.\]
i.e.\ this path only hits the first wall for $t= \tfrac{\sqrt{6}}{2}$ but it serves our purpose, in the sense that it provides us with a classical moduli description of the exceptional loci.

The contraction $g_4\colon  M \dashrightarrow X_3$ arises from the decomposition
\[ v= a_4 + b_4\ \text {with}\ a_4 = (-1,1,-2) = - v(\Ocal_S(-1))\ \text{and}\ b_4 = (1,1,1),\]
which is the only suitable decomposition in $\Hcal_4 =\langle a_4,v\rangle \cong (\Z^2,\left(\begin{smallmatrix}-2 & 3 \\ 3 & 8\end{smallmatrix}\right))$. Let $t> \tfrac{\sqrt{6}}{2}$. We note that $M_{\sigma_t}(a_4)$ consists of the point $\Ocal_S(-1)[1]$ and $S \cong M_{\sigma_t}(b_4)$ via $x \mapsto \Ical_x(1)$. Moreover, $\phi_t(a_4) > \phi_t(b_4)$. Hence the exceptional locus of $g_4$ consists of those sheaves $\Ecal \in M$ that arise as quotients
\[ 0 \rightarrow \Ocal_S(-1) \rightarrow \Ical_x(1) \rightarrow \Ecal \rightarrow 0 \]
for some $x \in S$. This is the projective bundle $Z_1$, as defined in Proposition \ref{Z1}.

Finally, there is the wall defined by $\Hcal_3 =\langle a_3,v\rangle \cong (\Z^2,\left(\begin{smallmatrix}-2 & 1 \\ 1 & 8\end{smallmatrix}\right))$, which only admits the decomposition
\[ v= a_3 + b_3\ \text {with}\ a_3 = (1,0,1)\ \text{and}\ b_3 = (-1,2,-2).\]
Let $\Ecal \in M\setminus Z_1$. Then $\Ecal$ is not destabilized at the first wall and thus $\Ecal \in X_3$. Now, $\Ecal$ is in the exceptional locus of $g_3 \colon X_3 \dashrightarrow X_2$ if and only if there is an exact triangle
\[ \Acal \rightarrow \Ecal \rightarrow \Bcal \xrightarrow{[1]},\]
where $\Acal \in M_{\sigma_t}(a_3)$ and $\Bcal \in M_{\sigma_t}(b_3)$ are stable objects and $1 < t < \tfrac{\sqrt{6}}{2}$. The space $M_{\sigma_t}(a_3)$ consists of the point $\Ocal_S$ and $S^{[3]} \cong M_{\sigma_t}(b_3)$, via $\Ical_{\xi} \mapsto R\sheafHom(\Ical_\xi,\Ocal_S)(-2)[1]$. Indeed, there are no walls for $S^{[3]}$ \cite[Prop 5.6]{Al} and $R\sheafHom(\Ical_\xi,\Ocal_S)(-2) \in \Coh^0(S)$. Consequently, $\Ecal$ is contracted if and only if there is $\xi \in S^{[3]}$ and an exact triangle
\[ \Ocal_S \rightarrow \Ecal \rightarrow R\sheafHom(\Ical_\xi,\Ocal_S)(-2)[1] \xrightarrow{[1]} \]
or equivalently an exact sequence
\[ 0 \rightarrow \Ocal_S(-2) \rightarrow \Ical_\xi \rightarrow \sheafExt^1(\Ecal(2),\Ocal_S) \rightarrow 0. \]
Here, $\Ecal' \coloneqq \sheafExt^1(\Ecal(2),\Ocal_S) \in M_H(0,2,-7)$. By the proof of Proposition \ref{Z1} we conclude that the exceptional locus of $g_3$ is $\tilde{Z}_3$.
\end{proof}

\begin{rem}
For $v=(0,2,-1)$, the wall-crossing can not be realized entirely in the $(H,H)$-plane. Actually, only two walls intersect the $(H,H)$-plane.
\end{rem}

\begin{cor}\label{Z1 ohne Z3}
We have that $Z_3\setminus Z_1$ is isomorphic to a $\Pbb^2$-bundle over an open subset of $S^{[3]}$. Moreover,
\[ Z_3 \cap \BN^1(M) = Z_1.\]
\end{cor}
\begin{proof}
The proof of Theorem \ref{birational models} implies that the map $\varphi_3 \colon \mathcal{X}_3 \dashrightarrow M_H(0,2,-7) \cong M$ from Proposition \ref{Z1} identifies the open subset of $\mathcal{X}_3$, where $\varphi_3$ is defined and injective with $Z_3\setminus Z_1$. Let $D \in |2H|$ be a smooth curve and $\Lcal \in \Pic^3(D)$ such that $h^0(\Lcal) \geq 2$. Then $\varphi_3$ is defined but not injective in $\Lcal^\vee$. Consequently, we must have $\Lcal \in Z_1$.
\end{proof}

We could have also determined $(g_3)^{-1} \colon X_2 \dashrightarrow X_3$.

\begin{prop}
We have
\[ \Bl_{\tilde{P_3}}X_2 \cong \Bl_{\tilde{Z_3}}X_3,\]
i.e.\ $\tilde{P}_3 \subset X_2$ is the dual projective bundle of $\tilde{Z}_3 \subset X_3$.
\end{prop}

\begin{proof}
We let $t_1 > r > t_2$ and $X_2 = M_{\sigma'_r}(v')$. We want to see that $\tilde{P_3}$ is the projective bundle parameterizing extensions of the form
\[ \Acal \rightarrow \Ical \rightarrow \Bcal \xrightarrow{[1]},\]
where $\Acal \in M_{\sigma'_r}(2,-2,1)$ and $\Bcal \in M_{\sigma'_r}(-1,2,-5)$. Necessarily, $\Bcal = \Ocal_S(-2)[1]$ and
\[S^{[3]} \cong M_{\sigma'_r}(2,-2,1),\ \Ical_\xi \mapsto \Ical_\xi(-1)\oplus \Ocal_S(-1).\]
Indeed, we verify that $\Ecal[1] \coloneqq (\Ical_\xi(-1)\oplus \Ocal_S(-1))[1] \in \Coh^{-2}(S)$. Assume that $\Fcal \subset \Ecal$ is a subbundle. Then either $\rk\Fcal = 2$ and $c_1(\Fcal)= c_1(\Ecal) = -2H$. Or $\rk\Fcal = 1$ and $\Fcal$ embeds into $\Ical_\xi(-1)$ or into $\Ocal_S(-1)$, which implies $c_1(\Fcal).H \leq -2$. In both cases, we have $\tfrac{H.c_1(\Fcal)}{\rk\Fcal} + 2\leq 0$ and hence $\Ecal[1] \in \Coh^{-2}(S)$.\\
Let $\Ical_\xi$ be the ideal sheaf of a generic point in $P_3 \subset S^{[5]}$. Then $\xi = \zeta \cup \xi'$ for a subscheme $\zeta \in S^{[2]}$ such that $\Supp(\zeta) = \{x,\iota(x)\}$ and $\xi' \in S^{[3]}$ is disjoint from $\zeta$ (cf.\ Example \ref{P}). The assumption on $\zeta$ is equivalent to $\Ical_\zeta(1)$ being globally generated with $h^0(\Ical_\zeta(1))=2$. This gives a short exact sequence
\begin{equation}\label{seq I}
0 \rightarrow \Ocal_S(-2) \rightarrow \Ocal_S(-1)^{\oplus 2} \rightarrow \Ical_{\zeta} \rightarrow 0.
\end{equation}
Therefore, $\Ical_\xi = \Ical_\zeta \cdot \Ical_{\xi'}$ fits into a diagram
\[ \xymatrix{
0 \ar[r]& \Ical_{\xi'}(-2) \ar[d]\ar[r]& \Ical_{\xi'}(-1)^{\oplus 2} \ar[d] \ar[r]& \Ical_{\xi}\ar@{=}[d] \ar[r]& 0\\
0 \ar[r]& \Ocal_S(-2) \ar[r]& \Ocal_S(-1)\oplus \Ical_{\xi'}(-1) \ar[r]& \Ical_{\xi} \ar[r]& 0.
}\]
Here, the first line is \eqref{seq I} tensored with $\Ical_{\xi'}$ and the second line is the pushout along the left vertical arrow. This finishes the proof.
\end{proof}

Interestingly, the dual projective bundle of $\tilde{W}_3 \subset X_1$ does not yield a component of the Brill--Noether locus in $M$.

\begin{prop}
Let $\tilde{W}_3^\vee \subset X_2$ be the exceptional locus of $(g_2)^{-1}\colon X_2 \dashrightarrow X_1$ and $B_3 \subset M$ its strict transform then
\[ B_3 \subset M_\Sigma.\]
More precicesly, let $D=D_1 \cup D_2 \in \Sigma\setminus \Delta$. Then
\[ B_3  \cap f^{-1}(D_1 \cup D_2) = \{ \Ecal \in f^{-1}(D_1 \cup D_2)\mid h^0(\Ecal|_{D_i}) \neq 0\ \text{for both}\ i=1,2 \}.\]
In particular, $B_3$ is not contained in $\BN^0(M)$ and $T^{-1}$ is generically defined in $B_3$.
\end{prop}

\begin{proof}
As above, we let $X_2 = M_{\sigma'_r}(v')$ with $t_1 > r > t_2$. We know that $\tilde{W}_3^\vee$ parameterizes extensions
\[ \Acal \rightarrow \Fcal \rightarrow \Bcal \xrightarrow{[1]},\]
where $\Acal \in M_{\sigma'_r}(1,0,-5)$ and $\Bcal \in M_{\sigma'_r}(1,-1,1)$. If $\Fcal \in \tilde{W}_3^\vee$ is a generic point, then we saw in the proof of Theorem \ref{birational models}, that we can take $\Acal = \Qcal \in M_H(0,1,-5)$ and $\Bcal = \Ical_x(-1)$ for a point $x\in S$. This gives
\[  \Qcal \rightarrow \Fcal \rightarrow \Ical_x(-1) \xrightarrow{[1]}.\]
To find the strict transform in $M$, we have to apply the spherical twist $T=T_{\Ocal_S(-2)}$ and tensor with $\Ocal_S(2)$. We find an exact triangle in $D^b(S)$
\begin{equation}\label{triangle twists}
T(\Qcal) \rightarrow T(\Fcal) \rightarrow T(\Ical_x(-1)) \xrightarrow{[1]},
\end{equation}
where $T(\Qcal)$ is concentrated in degree zero and fits into the sequence
\begin{equation}\label{T(Q)}
0 \rightarrow \Qcal \rightarrow T(\Qcal) \rightarrow \Ocal_S(-2) \rightarrow 0.
\end{equation}
Moreover, we compute that 
\[\Ocal_S(-2)^{\oplus 2} \rightarrow \Ical_x(-1) \rightarrow T(\Ical_x(-1)) \xrightarrow{[1]} \]
and thus $\Hcal^{-1}(T(\Ical_x(-1)))= \Ocal_S(-3)$ and $\Hcal^{0}(T(\Ical_x(-1)))= \Ocal_{\iota(x)}$.
Then the long exact cohomology sequence of \eqref{triangle twists} reads
\begin{equation}\label{T(F)}
0 \rightarrow \Hcal^{-1}(T(\Fcal)) \rightarrow \Ocal_S(-3) \rightarrow T(\Qcal) \rightarrow \Hcal^0(T(\Fcal)) \rightarrow \Ocal_{\iota(x)} \rightarrow 0.
\end{equation}
Generically, $T(\Fcal)$ is a Gieseker stable sheaf and in this case $T(\Fcal)(2)$ is a generic point in $B_3$. Combining \eqref{T(Q)} and \eqref{T(F)}, we see that if $T(\Fcal)$ is a sheaf, then $\Supp(T(\Fcal))= D_1 \cup D_2$, where $D_1 = \Supp(\Qcal)$ and $D_2 \in |H|$ such that $x \in D_2$. Hence $B_3 \subset M_\Sigma$.

Now, if we assume that $x \notin D_1$, then $\Ecal \coloneqq T(\Fcal)(2)$ is an extension
\begin{equation}\label{ext E}
0 \rightarrow \Qcal\otimes \Ocal_S(2)|_{D_1} \rightarrow \Ecal \rightarrow \Ocal_{D_2}(\iota(x)) \rightarrow 0.
\end{equation}
As $\sheafTor_1^{\Ocal_{{D_1} \cup D_2}}(\Ocal_{D_1},\Ocal_{D_2}(\iota(x)))=0$, this implies $\Ecal|_{D_1} \in \Pic^2({D_1})$, which is trivially effective. Moreover, restricting \eqref{ext E} to $D_2$ gives
\[ 0 \rightarrow \sheafTor_1^{\Ocal_{{D_1} \cup D_2}}(\Ocal_{D_2},\Ocal_{D_2}(\iota(x))) \xrightarrow\sim \Ocal_{{D_1} \cap D_2} \xrightarrow  0 \Ecal|_{D_2} \xrightarrow \sim \Ocal_{D_2}(\iota(x)) \rightarrow 0.\]
 In particular, $\Ecal|_{D_2} $ is an effective line bundle of degree one. Conversely,
\[\dim \{ \Ecal \in f^{-1}(D_1 \cup D_2)\mid h^0(\Ecal|_{D_i}) \neq 0\ \text{for both}\ i=1,2 \} = 4.\]
Hence, by dimension reasons, $B_3 \cap f^{-1}(D_1 \cup D_2)$ is as claimed for every $D=D_1 \cup D_2 \in \Sigma \setminus \Delta$.

It is left to show, that for a generic sheaf $\Ecal \in B_3$, we have $h^0(\Ecal)= 0$. To see this, we can assume $\Ecal \in \Pic^{(1,2)}(D_1\cup D_2)$ and that $h^0(\Ecal|_{D_i}) = 1$ for $i=1,2$. If $h^0(\Ecal)\neq 0$, necessarily $h^0(\Ecal)=1$ and the restriction to each component induces an isomorphism on global sections. However, this determines $\Ecal$ completely. Indeed, we have $\Ecal|_{D_1} = \Ocal_{D_1}(x)$ and $\Ecal|_{D_2} = \Ocal_{D_2}(y+z)$, for unique points $x,y$ and $z$. Then any non-zero section $\Ocal_{D_1 \cup D_2} \rightarrow \Ecal$ is necessarily injective with cokernel supported on $\xi = \{x,y,z\}$. In other words, $\Ecal$ has a section if and only if $\Ecal^\vee = \ker(\Ocal_D \rightarrow \Ocal_\xi)$. Finally, we saw in Proposition \ref{birational map M S5}, that $T^{-1}$ is defined for all $\Ecal \in M_{\Sigma\setminus\Delta}$ such that $h^0(\Ecal)=0$.
\end{proof}

\end{document}